\documentclass[12pt,letter]{amsart}
\usepackage[text={6.5in,9in},centering]{geometry}

\usepackage{amsmath,amssymb}
\usepackage{latexsym}
\usepackage{mathrsfs}
\usepackage{amsthm}
\usepackage{graphicx}
\usepackage{color}
\usepackage{colortbl}
\usepackage[config,labelfont={rm,rm},textfont=rm]{caption,subfig}

\newtheorem{theorem}{Theorem}[section]
\newtheorem{teo}{Theorem}[section]
\newtheorem{lemma}[theorem]{Lemma}

\newtheorem{remark}[theorem]{Remark}

\usepackage[latin1]{inputenc}
\usepackage[english]{babel}

\newcommand{\calN}{ \mathcal{N}}

\renewcommand{\a}[2]{a\,(#1,#2)}
\newcommand{\ah}[2]{a_{h}\,(#1,#2)}

\newcommand{\Pp}{ \mathbb{P}}

\newcommand{\K}{K} 
\renewcommand{\O}{\Omega} 
\newcommand{\Eh}{{{\mathcal E}_h}} 
\newcommand{\Eho}{{{\mathcal E}^{o}_h}} 

\newcommand{\Ehb}{{{\mathcal E}^{\partial}_h}} 

\newcommand{\Th}{\mathcal{T}_h} 
\newcommand{\hK}{\h_{\K}} 
\newcommand{\h}{h} 

\newcommand{\bx}{{\bf x}}

\newcommand{\uh}{u_h} 
\newcommand{\vh}{v_h}

\newcommand{\ui}{u^{I}}

\newcommand{\n}{{\bf n}} 

\newcommand{\dx}{\,\mbox{d}x} 
\newcommand{\jump}[1]{\lbrack\!\lbrack\,#1\,\rbrack\!\rbrack} 

\newcommand{\re}{ \mathbb{R}}

\numberwithin{equation}{section}
\numberwithin{figure}{section}
\numberwithin{table}{section}

\renewcommand\lll{|\kern-1pt|\kern-1pt|} 

\newcommand{\vect}[1]{\boldsymbol{#1}}

\renewcommand{\lor }{\longrightarrow}

\newcommand{\triplenorm}[1]{%
  \left\vert\kern-0.9pt\left\vert\kern-0.9pt\left\vert #1
  \right\vert\kern-0.9pt\right\vert\kern-0.9pt\right\vert}

\newcommand{\hh}{h}
\newcommand{\kk}{k}

\newcommand{\MFD}{\textsc{MFD}}
\newcommand{\VEM}{\textsc{VEM}}

\newcommand{\ASSUM}[2]{\protect{(\textsf{#1}{#2})}}

\newcommand{\Vhk} {V_\hh^\kk}
\newcommand{\Vhkg}{V_{\hh,g}^\kk}

\newcommand{\dS}{\,ds}

\newcommand{\mK}{\abs{\K}}
\newcommand{\mE}{\abs{\E}}
\newcommand{\hE}{\hh_{\E}}
\newcommand{\E}{e}

\newcommand{\bxE}{\bx_{\E}}
\newcommand{\bxK}{\bx_{\K}}
\newcommand{\nE} {\n_{\E}}

\newcommand{\aK} [2]{a^{\K}(#1,#2)}
\newcommand{\aKh}[2]{a_{\hh}^{\K}(#1,#2)}
\newcommand{\SK} [2]{S^{\K}(#1,#2)}

\newcommand{\norm}[1]{\|#1\|}
\newcommand{\snorm}[1]{\abs{#1}}
\newcommand{\abs}[1]{|#1|}

\newcommand{\dual}[2]{\langle #1, #2\rangle}

\newcommand{\mm}{m}
\newcommand{\momE}[1]{\mu_{\E}^{#1}}
\newcommand{\momK}[1]{\mu_{\K}^{#1}}

\newcommand{\Hnc}{H^{1,\textrm{nc}}(\Th;\kk)}

\newcommand{\npoly}   {n_{\K,\kk}}
\newcommand{\sizeVhk} {N_{\K}}
\newcommand{\sizeVhkg}{N_{\Th}}

\newcommand{\NE}{N_{\text{edges}}}
\newcommand{\NF}{N_{\text{faces}}}
\newcommand{\NK}{N_{\text{elements }}}

\newcommand{\vI}{v^{I}}
\newcommand{\fh}{f_{\hh}}

\renewcommand{\wp}{w_{\pi}}

\newcommand{\PINK}{\Pi^{\nabla}_{\K}}

\newcommand{\bPI} {\vect{\Pi}}
\newcommand{\bPIp}{\vect{\Pi}^{\perp}}
\newcommand{\bPIN}{\bPI^{\nabla}}

\newcommand{\PAR}[1]{\medskip\noindent%
  \textbf{\emph{#1.}}
}

\newcommand{\matB}{\mathsf{B}}

\newcommand{\matD}{\mathsf{D}}
\newcommand{\matG}{\mathsf{G}}
\newcommand{\matI}{\mathsf{I}}
\newcommand{\matM}{\mathsf{M}}

\newcommand{\matS}{\mathsf{S}}
\newcommand{\matU}{\mathsf{U}}
\newcommand{\matGt}{\widetilde{\mathsf{G}}}

\newcommand{\calM}{ \mathcal{M}}

\title{The nonconforming virtual element method}

\author[B. Ayuso de Dios]{Blanca Ayuso de Dios}
\address{Center for Uncertainty Quantification in Computational Science \& Engineering\\
Computer,~Electrical and Mathematical Sciences \& Engineering Division (CEMSE) \\
King Abdullah University of Science and Technology\\
Thuwal 23955-6900, Kingdom of Saudi Arabia}
\author[K. Lipnikov]{K. Lipnikov}
\address{Applied Mathematics and Plasma Physics Group, Theoretical Division,
  Los Alamos National Laboratory, Los Alamos, NM 87545, USA;}
  
\author[G. Manzini]{Gianmarco Manzini}
\address{Applied Mathematics and Plasma Physics Group, Theoretical Division,
  Los Alamos National Laboratory, Los Alamos, NM 87545, USA;  Istituto di Matematica Applicata e Tecnologie
  Informatiche (IMATI) -- CNR, via Ferrata 1, I -- 27100 Pavia, Italy.}

\begin{document}

\maketitle

\begin{abstract}
  We introduce the {\it nonconforming} Virtual Element Method (\textsc{VEM}) for
  the  approximation of  second order elliptic problems. 
  We present the construction of the new element in two and three dimensions, highlighting the main
  differences with  the conforming \textsc{VEM}  and the classical 
  nonconforming finite element methods.
  We provide the error analysis and establish the equivalence with a family of mimetic finite difference methods.
\end{abstract}



\section{Introduction}\label{sec:0}

Methods that can handle general meshes consisting of arbitrary polygons 
or polyhedra have received significant attention over the last years.
Among them  the mimetic finite difference (MFD) method that has been  
successfully applied to a wide
range of scientific and engineering applications (see, for 
instance,~{\cite{Cangiani-Manzini-Russo:2009, Lipnikov-Manzini-Brezzi-Buffa:2011, Palha-Rebelo-Hiemstra-Kreeft-Gerritsma:2014-JCP, Lipnikov-Manzini-Shashkov:2014-JCP}} and the references therein).
However, the construction of high-order MFD schemes is still a challenging 
task even for two and three dimensional second order elliptic problems.
For example, the two-dimensional MFD scheme in \cite{BeiraodaVeiga-Lipnikov-Manzini:2011} could 
be seen as the high-order extension of the lower-order scheme
given in \cite{Brezzi-Buffa-Lipnikov:2009}. 
A straightworward extension of \cite{BeiraodaVeiga-Lipnikov-Manzini:2011} to three dimensions would 
lead to a clumsy discretization involving a huge number of degrees of freedom that ensure 
conformity of the approximation. 
By relaxing the conformity condition, a simpler MDF scheme
has been proposed in \cite{Lipnikov-Manzini:2014:JCP} for three dimensional elliptic problems. 
 
Very recently, in the pioneering work \cite{volley}, the basic principles of the 
{\it virtual element method} (VEM) have been introduced. 
The VEM allows one to recast the MFD
schemes \cite{Brezzi-Buffa-Lipnikov:2009,BeiraodaVeiga-Lipnikov-Manzini:2011}
as Galerkin formulations.  The virtual element methodology generalizes the classical finite element
method to mesh partitions consisting of polygonal and polyhedral elements 
of very general shapes including non-convex elements. In this
respect, it shares with the MFD method the flexibility of mesh
handling. Unlike the MFD method, the VEM provides a sound
mathematical framework that allows to devise and analyze new
schemes in a much simpler and elegant way.  
The name {\it virtual} 
comes from the fact that the local approximation space in each
mesh polygon or polyhedra contains a space of
polynomials together with some other functions that are solutions
of particular partial differential equations. Such functions are
never computed and similar to the MFD method, the VEM can be implemented
using only the degrees of freedom and the polynomial part of the approximation space. 
We refer to
\cite{BeiraodaVeiga-Brezzi-Marini-Russo:2014} for the implementation details.

Despite of its infancy, the conforming VEM laid in \cite{volley} has  
been already extended to a variety of two dimensional problems:
plate problems are studied in \cite{vem-plates}, linear elasticity 
in \cite{vem-elas},  mixed methods for $H(\mbox{div};\O)$-approximations 
are introduced in \cite{mixedVEM}, and very recently the VEM has been 
extended to simulations on discrete fracture networks \cite{berrone}.
In \cite{projettori}, further tools are presented that allow us to construct 
and analyze the conforming VEM for three dimensional elliptic problems.
The definition of the three dimensional virtual element spaces in  \cite{projettori}, requires the use of the two dimensional ones.  

In this paper, we develop and analyze the {\it nonconforming} VEM for the approximation 
of second order elliptic problems in two and three dimensions. 
We show that the proposed method contains the MFD schemes from \cite{Lipnikov-Manzini:2014:JCP}. 
In contrast to the conforming VEM, our construction is done simultaneously 
for any dimension and any approximation order.
To put this work in perspective, we present below a brief (non exhaustive) overview of 
nonconforming finite element methods.
 
\subsection{Overview of nonconforming finite element methods} 
Nonconforming finite elements were first recognized as a {\it variational crime}; 
a term first coined by Strang in \cite{strang0,strang1}. 
In the case of second order elliptic problems, the approximation space has some
continuity built in it, but still discrete functions are not continuous. 
Still, that relaxed continuity (or crime) has proved its usefulness in many 
applications, mostly related to continuum mechanics,  
in particular, for fluid flow problems \cite{heyran, rana-turek} (for moderate 
Reynolds numbers) and elasticity \cite{falk91,ortner}. 

The construction, analysis and understanding of nonconforming elements have received 
much attention since their first introduction for second order elliptic problems. 
In two dimensions, the design of schemes of order of accuracy $k\geq 1$ was guided 
by the patch-test, which enforces continuity at $k$ Gauss-Legendre points on edge.
Due to different behavior of odd and even polynomials, the construction of schemes 
for odd and even $k$ is different, with the latter case demanding much more elaborated arguments. 
Furthermore, the shape of the elements (triangular or rectangular/quadrilateral) adds additional
complexity to the construction of nonconforming elements \cite{rana-turek} (the result of having an odd or even number of edges in the element leads to a different construction). 
For the Stokes problem with the Dirichlet boundary conditions, Crouzeix and Raviart  
proposed and analyzed the first order ($k=1$) nonconforming finite element approximation of 
the velocity field in \cite{cr-stokes}, which is now know as the Crouziex-Raviart element. 
The extension to degree $k=3$ was given in \cite{Crouzeix-Falk:1989}, while the construction 
for degree $k=2$ was introduced in \cite{Fortin-Soulie:1983}. 
In all cases, the inf-sup stable Stokes pair is formed by considering discontinuous 
approximation for the pressure of one degree lower.

Already in the 80's, an equivalence between mixed methods and a modified version
of nonconforming elements of odd degree has been established in \cite{arnold-franco,dona00} 
and exploited in the analysis and implementation of the methods.  
The author of \cite{comodi}, inspired by \cite{arnold-franco}, has studied 
the hybridization of the mixed Hellan-Herrmann-Johnson method (of any degree) for 
the approximation of a fourth order problem. 
As a byproduct of the analyzed postprocessing technique (that uses the gradient of the 
displacement which would play the role of the velocity field in the Stokes problem),
a construction of nonconforming elements of any degree $k$ is provided. 
Again, this construction distinguishes  between odd and even degrees. Although the details are for the fourth order problem, the strategy can be adapted to other elliptic problems.
For $k=1,3$, the nonconforming element coincides with 
the construction given in \cite{cr-stokes,Crouzeix-Falk:1989} for the Stokes problem. 
For even $k$, in addition to the moments of order $k-1$ on each edge, an extra degree of freedom is required to ensure unisolvence. In the case $k=2$ the resulting  nonconforming local  finite element space has the same dimension 
but is different to the one proposed in \cite{Fortin-Soulie:1983} where it is 
constructed by adding a nonconforming bubble to the second order conforming space.

Over the last years, further generalizations of the nonconforming elements have been 
still considered by several authors; always distinguishing between odd and even degrees. 
In  \cite{Stoyan-Baran:2006,Baran-Stoyan:2007} a construction similar to the one given in \cite{comodi} is considered for the Stokes problem. A rather different approach is considered in \cite{Matthies-Tobiska:2005}.
Finally, while the extension to three dimensions is simple for $k=1$, already for $k=2$, the construction 
of a nonconforming element becomes cumbersome \cite{Fortin:1985}.

\subsection{Main contributions}
In this paper, we extend the virtual element methodology by developing in {\it one-shot} 
(no special cases) a nonconforming approximation of any degree for any spatial dimension 
and any element shape.
For triangular meshes and $k=1,2$, the proposed nonconforming VEM has the same degrees 
of freedom as the related nonconforming finite element in \cite{comodi}. 
For quadrilaterals and $k=1$, the degrees of freedom are the same as that in \cite{rana-turek}.
The three main contributions of the present work are as follows.

{(i)} The nonconforming VEM is constructed for any
order of accuracy and for arbitrarily-shaped polygonal or/and polyhedral
elements. It also provides a simpler construction on simplicial meshes and quadrilateral meshes.
 
{(ii)} Unlike the conforming VEM \cite{proiettori}, the nonconforming VEM 
is introduced and analyzed at once for two and three dimensional problems. 
This simplifies substantially its analysis and practical implementation.

{(iii)} We prove optimal error estimates in the energy norm and (for $k\geq 2$) 
in the $L^{2}$-norm. The analysis of the new method is carried out using techniques 
already introduced in \cite{volley, vem-elas} and extending the results well known in the 
classical finite elements to the virtual approach.
As the byproduct of our analysis, we provide the theory for the MFD schemes 
in \cite{Lipnikov-Manzini:2014:JCP}.
 
To convey the main idea of our work in a better way and to keep the presentation simple, 
we consider the Poisson problem. However, all results apply (with minor changes) 
to more general second order problems with constant coefficients.

 The outline of the paper is as follows.
In Section~\ref{sec:1} we formulate the problem and
introduce the basic  setting.
In Section~\ref{sec:2} we introduce the nonconforming
VEM.
Section~\ref{sec:3} is devoted to the error analysis of
the nonconforming approximation.
In Section~\ref{sec:4} we establish the connection with
the nonconforming MFD method proposed
in~\cite{Lipnikov-Manzini:2014:JCP}.
In Section~\ref{sec:5} we offer some final remarks and
discuss the perspectives for future work and developments.

\section{Continuous problem and basic setting}
\label{sec:1}

In this section we present the basic setting and describe the continuous problem. To ease the presentation and give a clear view of the ideas we consider the Poisson problem the simple Poisson problem. However it is worth stressing that  all the results in the present paper extend straightforwardy to more general second order problems with piecewise constant coefficients. 

{\sc Notation:} Throughout the paper, we use the standard notation of Sobolev
spaces, cf.~\cite{Adams75}.
Moreover, for any integer $\ell\ge 0$ and a domain $D\in\re^{m}$ with
$m\leq d$, $d=2,3$, $\Pp^{\ell}(D)$ is the space of polynomials of
degree  at most $\ell$ defined on $D$. We also adopt the convention that $\Pp^{-1}(D)=\{0\}$.

\subsection{Continuous problem}
Let the domain $\Omega$ in $\re^{d}$ with $d=2,3$ be a bounded
open polytope with boundary $\partial\Omega$, e.g., a polygonal domain
with straight boundary edges for $d=2$ or a polyhedral domain with
flat boundary faces for $d=3$.
Let $f$ be in $L^2(\Omega)$ and consider the simplest model problem:
\begin{align}
  -\Delta u &= f\phantom{0} \quad\textrm{in~}\Omega,        \label{eq:mod00:A}\\
  u         &= g\phantom{f} \quad\textrm{on~}\partial\Omega.\label{eq:mod00:B}
\end{align}

Let $V_g=\{v\in H^1(\Omega)\,:\,v|_{\partial\Omega}=g\}$ and
$V=H^{1}_{0}(\Omega)$.
The variational formulation of problem
\eqref{eq:mod00:A}-\eqref{eq:mod00:B} reads as
\begin{equation}\label{eq:var00}
  \mbox{Find $u\in V_g$ such that:}\quad\a{u}{v}=\dual{f}{v}\qquad\forall v\in V,
\end{equation}
where the bilinear form $a:V\times V\to\re$ is given by
\begin{align}
  \a{u}{v}=\int_{\Omega}\nabla u\cdot\nabla v\dx \qquad\forall u, v\in V,
  \label{eq:exact:a}
\end{align}
and $\dual{\cdot}{\cdot}$ denotes the duality product between the
functional spaces $V'$ and $V$.
The bilinear form in~\eqref{eq:exact:a} is  continuous and coercive with respect to the $H^{1}_0$-seminorm (which is a norm in $V$ by Poincare inequality); therefore  Lax-Milgram theorem ensures the well posedness of the variational
problem and, therefore, the existence of a unique solution $u\in V$ to
\eqref{eq:var00}.


\subsection{Basic setting}
\label{subsec:basic:setting}
We describe now the basic assumptions of the mesh partitioning and introduce some further functional spaces.
 
Let $\{\Th\}_h$ be a family of decompositions $\O$ into elements, $\K$ and let $\Eh$ denote the skeleton of the partition, i.e., the set of edges/faces of $\Th$. By $\Eho$ and $\Ehb$  we will refer to the set of interior and boundary edges/faces, respectively.
Following \cite{volley,proiettori} we make the following assumptions on the family of partitions:\\

{\bf (A0)} {\bf Assumptions on the family of partitions $\{\Th\}_h$}:  we assume that there exists a positive $\varrho>0$ such that
\begin{itemize}
\item for every element $\K$ and for every edge/face $e \subset \partial\K$, we have:
  $h_e\geq \varrho h_{\K}$\;,
\item every element $\K$ is star-shaped with respect to all the points of a sphere of
radius $\ge \varrho h_{\K}$;
\item for $d=3$, every face $e\in \Eh$ is star-shaped with respect to all the points of a disk having
radius $\ge \varrho h_{e}$.
\end{itemize}
The maximum of the diameters of the elements $\K\in \Th$ will be denoted by $h$. For every $h>0$, the partition $\Th$ is made of a finite number of polygons or polyhedra.\\

We introduce the broken Sobolev space for any $s> 0$
\begin{equation*}
  H^{s}(\Th) = \prod_{\K\in\Th} H^{s}(\K)=\big\{\,v\in L^{2}(\Omega)\,\,:\,\,\,v|_{\K}\in H^{s}(\K)\,\big\}, \qquad s> 0,
\end{equation*}
and define the broken  $H^{s}$-norm
\begin{equation*}
\|v\|^{2}_{s,\Th}:= \sum_{\K\in \Th} \| v\|_{s,\K}^{2} \qquad
\forall\, v \in H^{s}(\Th)\;,
\end{equation*}
and for $s=1$  the broken  $H^{1}$-seminorm
\begin{equation}\label{eq:norm-broken}
|v|^{2}_{1,h}:= \sum_{\K\in \Th} \|\nabla v\|_{0,\K}^{2} \qquad
\forall\, v \in H^{1}(\Th)\;.
\end{equation}
Let $e\subset \partial\K^{+} \cap \partial\K^{-}$ be an edge/face in $\Eho$. For $v\in H^{1}(\Th)$, by $v^{\pm}$ we denote the trace of $v|_{\K^{\pm}}$ on $e$ taken from within the element $\K^{\pm}$ and by $n_{e}^{\pm}$ we denote the unit normal on $e$ in the outward direction with respect to $\K^{\pm}$. We  then define the jump operator as:
\begin{equation}\label{eq:def-jump}
\jump{v}:=v^{+}\n_e^{+} +v^{-}\n_e^{-} \quad \mbox{on   } e\in \Eho \quad\textrm{and}\quad  \jump{v}:=v\n_e \mbox{     on } e\in \Ehb\;,
\end{equation}
where on boundary edges/faces we have defined it as the normal component of the trace of $v$.

It is convenient to introduce a space (subspace of  $H^{1}(\Th)$) with some continuity built in. For an integer  $k\geq 1$, we define
\begin{equation}\label{eq:h1:02b}
  \Hnc = \left\{\,
    v\in H^{1}(\Th)\,:\,\int_{\E}\jump{v}\cdot\nE\,q\dS=0\,
    \,\,\forall\,q\in\Pp^{k-1}(\E),
    \,\,\forall\E\in\Eh
    \,\right\}.
\end{equation}

 Although for discontinuous functions $|\cdot|_{1,h}$ is only a semi-norm,  for $v\in V=H^{1}_0(\O)$ and  $v\in H^{1,\text{nc}}(\Th)$ it is indeed a norm. In fact,  a standard application of the results in \cite{Brenner03} shows that a Poincare inequality holds for functions in $H^{1,nc}(\Th)$ (already with $k=1$),  i.e, there exists a constant $C_{P}>0$ independent of $h$ such that
\begin{equation}\label{eq:poinc:02}
  \|v\|_{0,\Th}^{2} \leq C_{P} |v|_{1,h}^{2} \quad \forall\, v\in H^{1,\text{nc}}(\Th)\;.
\end{equation}
Therefore with a small abuse of notation we will refer to the broken semi-norm as a norm.

\begin{remark}
  The space $H^{1,\text{nc}}(\Th;1)$ (i.e., $k=1$), is the space with
  minimal required continuity to ensure that the analysis can be
  carried out.
\end{remark}

Finally, the bilinear form $a(\cdot,\cdot)$  can be split as:
\begin{equation}\label{eq:split0}
  \a{u}{v}=\sum_{\K\in\Th}\aK{u}{v}
  \quad\textrm{where}\quad
  \aK{u}{v}=\int_{\K}\nabla u\cdot\nabla v\dx
  \quad\forall u,v\in V.
\end{equation}

\section{Non-conforming virtual element method}
\label{sec:2}

In this section we introduce the nonconforming virtual finite element
method for the model problem \eqref{eq:mod00:A}-\eqref{eq:mod00:B}
which we will write as a Galerkin approximation:

\begin{equation}\label{eq:var0h}
  \mbox{Find $\uh\in\Vhkg$ such that:}\quad\ah{\uh}{\vh}=\dual{\fh}{\vh}\qquad\forall\vh\in\Vhk,
\end{equation}
where $\Vhk\subset\Hnc$ is the global nonconforming virtual space;
$\Vhkg$ is the affine space required by the numerical treatment of the
Dirichlet boundary conditions, $a_{\hh}$ and $\dual{\fh}{\cdot}$ are
the nonconforming approximation to the bilinear form $a$ and the
linear functional $\dual{f}{\cdot}$.\\

We start by describing the local and global nonconforming virtual finite element spaces (denoted by $V_h^{k}(\K)$ and $V_h^{k}$ respectively). We then construct the discrete bilinear form ($a_h(\cdot,\cdot)$)  and right hand side ($f_h$),  discussing also their main properties for the analysis of the resulting approximation.
Throught the whole section, we follow the basic ideas given in \cite{volley,proiettori}, trying to highlight the main differences in the present case.

\subsection{The local nonconforming virtual element space $\Vhk(\K)$}
We need to introduce some further notation. For a simple polygon or polyhedra $\K$ with $n$ edges/faces we denote by $x_{\K}$ its center of gravity, by $|\K|$ its $d$-dimensional measure (area for $d=2$, volume for $d=3$) and by $h_{\K}$ the diameter of the element $\K$. Similarly, for each edge/face $e\subset \partial\K$, we denote by $x_{e}$ the midpoint/barycenter of the edge/face, by $|e|$ its measure and by $h_e$ the diameter of the edge/face. As before, $\n_{\K}$ denotes the outward unit normal on $\partial\K$ and $\n_{e}$ refers to the outward unit normal on $e$.

We define for $k\geq 1$ the finite dimensional space $V^{k}_h(\K)$ associated to the polygon/polyhedra $\K$:

\begin{equation}\label{eq:def:Vhk0}
  \Vhk(\K)=\big\{\,
  v\in H^{1}(\K)\,:\,\frac{\partial v}{\partial \n}\in\Pp^{k-1}(\E)
  \,\,\,\forall\E\subset\partial\K,\,\,\,
  \Delta v\in\Pp^{k-2}(\K)\,
  \big\}\;,
\end{equation}
with the usual convention that $\mathbb{P}_{-1}(\K)=\{0\}$.\\

For $k=1$, $V^{1}_h(\K)$ consists of functions $v$ whose normal derivative $\frac{\partial v}{\partial\n}$ is constant on each $e\subset \partial\K$  (and different on each $e$) and inside $\K$ are harmonic (i.e., $\Delta\,v=0$). This characterization seems to give $n+1$ conditions, but a closer
look reveals that we are precisely imposing $n$ conditions. The reason is that a harmonic function in $V^{1}_h(\K)$ can be uniquely determined by adding the solvability condition $\int_{\partial\K}\frac{\partial v}{\partial \n} =0$   which gives $\sum_{e\in\partial\K} \int_{e}\frac{\partial v}{\partial \n} =0$ and reduces by $1$ the number of conditions, hence $n+1-1=n$.\\

For $k=2$, the space $V_{\hh}^{2}(\K)$ is made of functions for which
the normal derivative along the edges/faces $\E\in\partial\K$ is a
linear polynomial and inside $\K$,  $\Delta v$ is constant. A simple counting reveals then that the dimension of $V^{2}_{\hh}$ is $dn+1$.

For each polygon/polyhedra $\K$, the dimension of $V_h^{k}(\K)$ is given by
\begin{equation}\label{eq:Nk}
  \sizeVhk =
  \begin{cases}
    nk + (\kk-1)\kk/2                 & \mbox{for~} d=2,\\[0.25em]
    nk(\kk+1)/2 + (\kk-1)\kk(\kk+1)/6 & \mbox{for~} d=3.
  \end{cases}
\end{equation}

Let $s=(s_1,\ldots,s_d)$ be a $d$-dimensional multi-index with the
usual notation that $\abs{s}=\sum_{i=1}^{d}s_i$ and
$\bx^s=\prod_{i=1}^{d}x_i^{s_i}$ where $\bx=(x_1,\ldots,x_d)\in\re^d$.
For $\ell\geq 0$, the symbols $\calM^{\ell}(\E)$ and
$\calM^{\ell}(\K)$ respectively denote the set of \emph{scaled monomials}
on $\E$ and $\K$:
\begin{equation}
  \calM^{\ell}(\E)=\left\{\left(\frac{\bx-\bxE}{\hE}\right)^{s},\,\,\abs{s}\leq\ell\right\}
  \quad\textrm{and}\quad
  \calM^{\ell}(\K)=\left\{\left(\frac{\bx-\bxK}{\hK}\right)^{s},\,\,\abs{s}\leq\ell\right\}.
\end{equation}

\begin{figure}[!t]
  \centering
  \begin{tabular}{cccc}
    \includegraphics[width=0.2\textwidth]{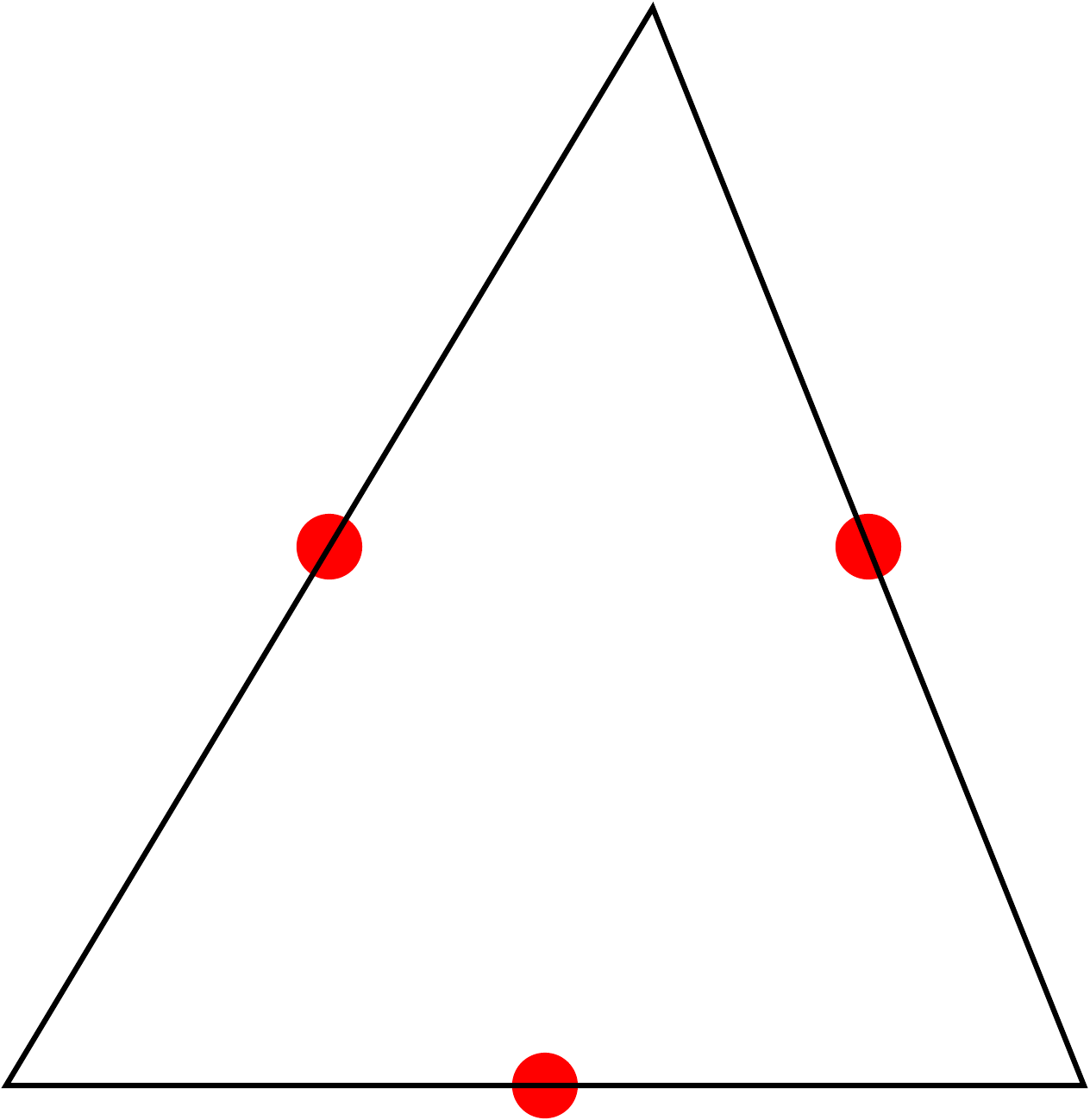} &\quad
    \includegraphics[width=0.2\textwidth]{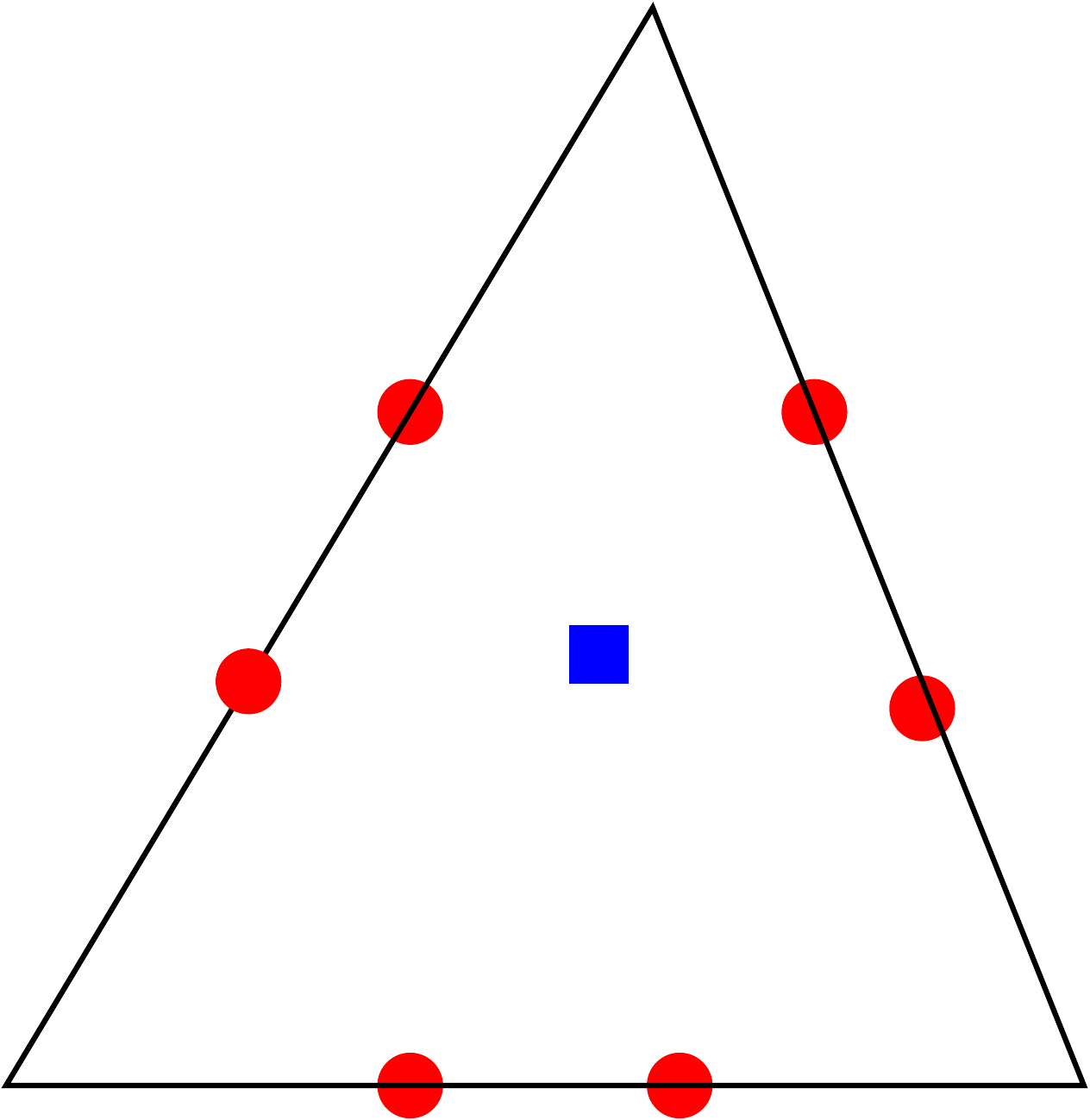} &\quad
    \includegraphics[width=0.2\textwidth]{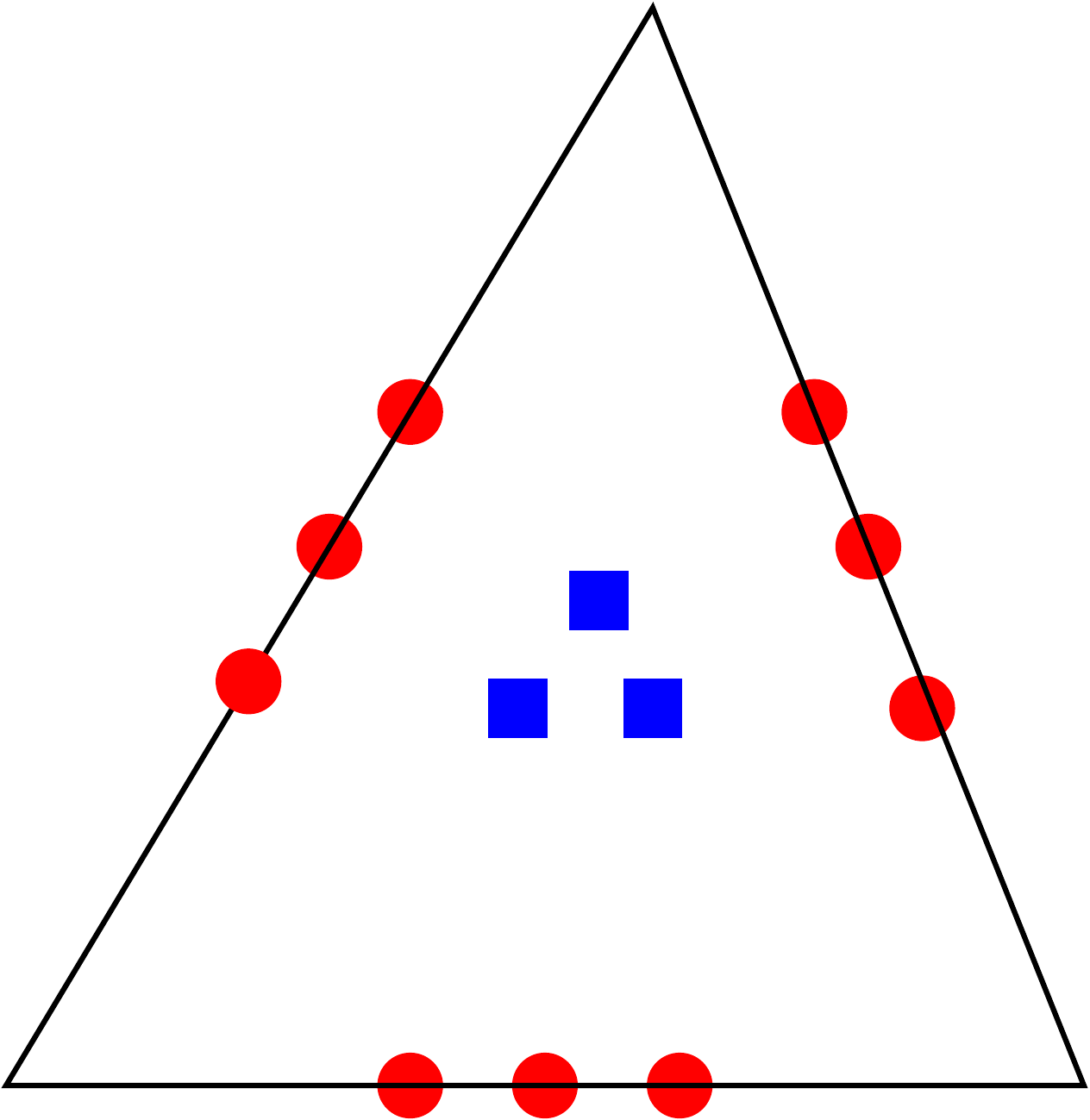} &\quad
    \includegraphics[width=0.2\textwidth]{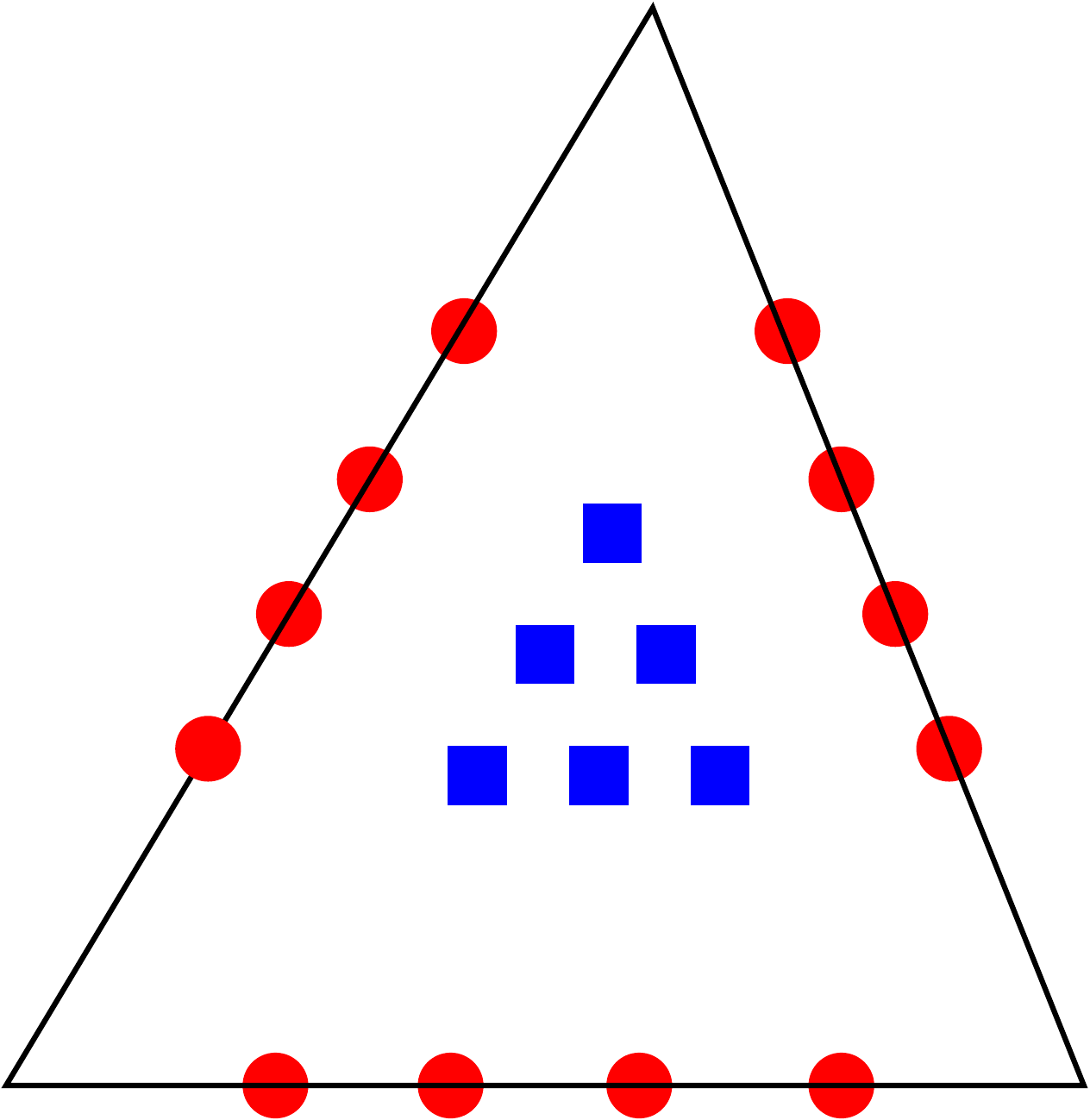} \\
    $\mathbf{k=1}$ & $\mathbf{k=2}$ & $\mathbf{k=3}$ & $\mathbf{k=4}$
  \end{tabular}
  \caption{Degrees of freedom of a triangular cell for $k=1,2,3,4$; edge
    moments are marked by a circle; cell moments are marked by a square.}
  \label{fig:dofs:triangle}
\end{figure}

\begin{figure}[!t]
  \centering
  \begin{tabular}{cccc}
    \includegraphics[width=0.2\textwidth]{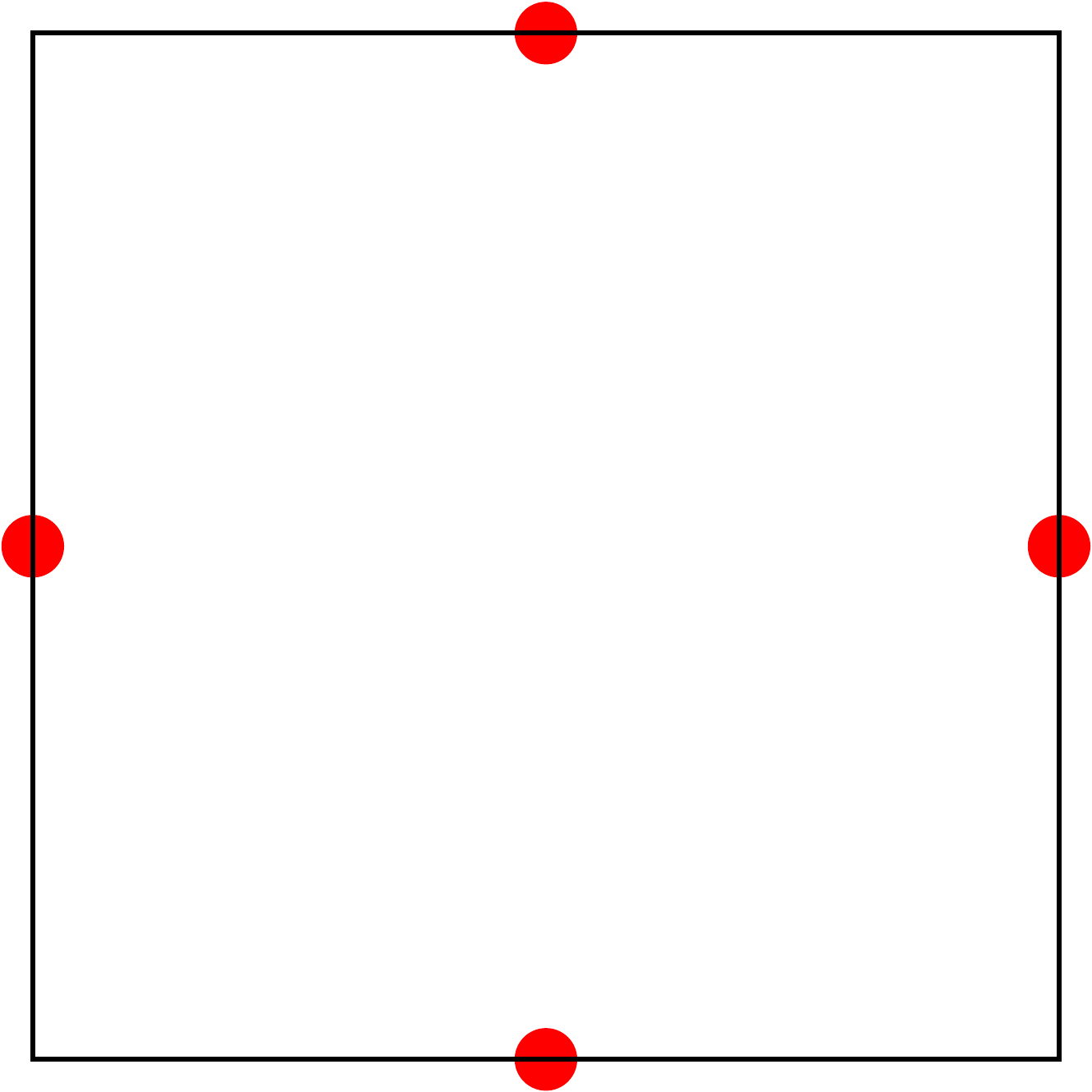} &\quad
    \includegraphics[width=0.2\textwidth]{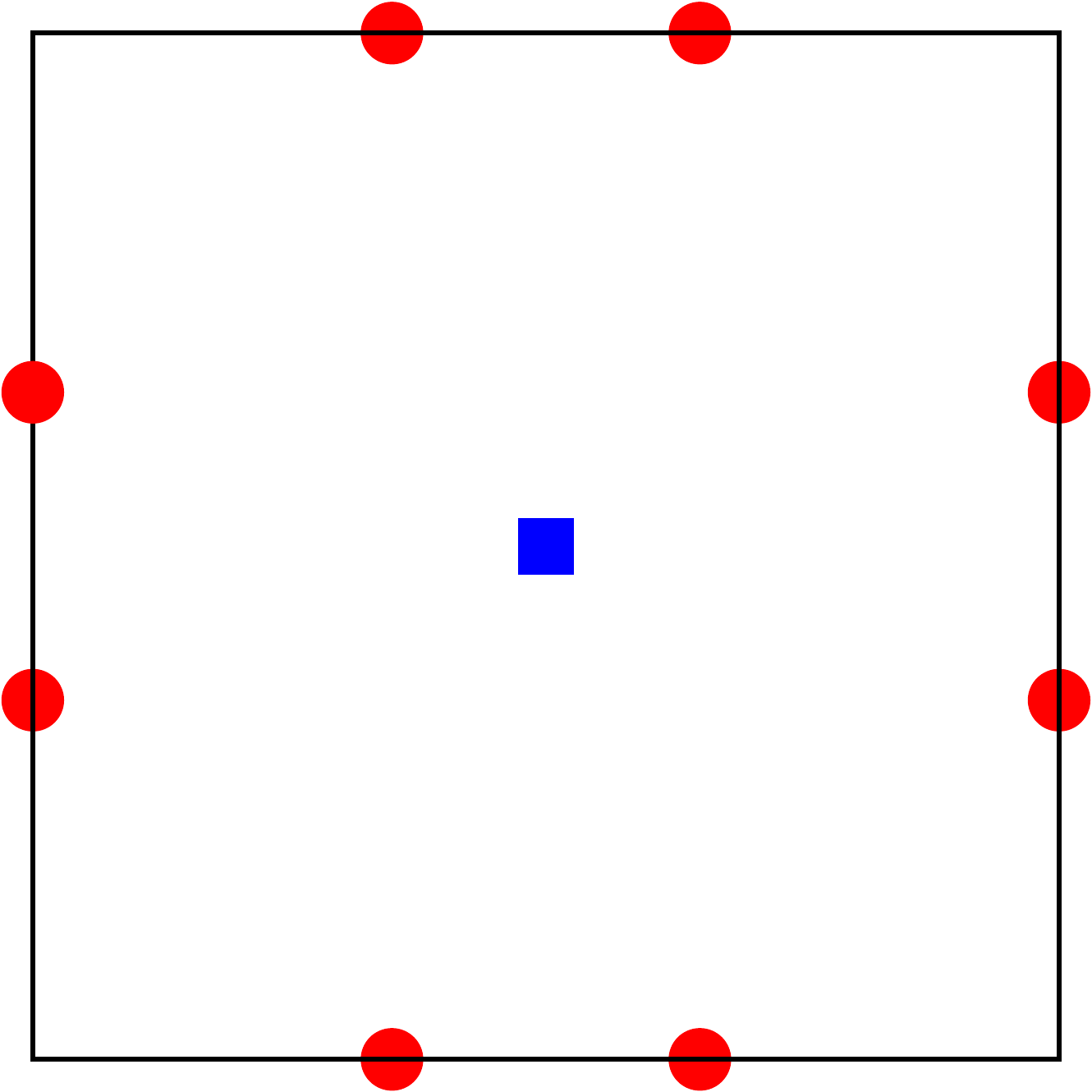} &\quad
    \includegraphics[width=0.2\textwidth]{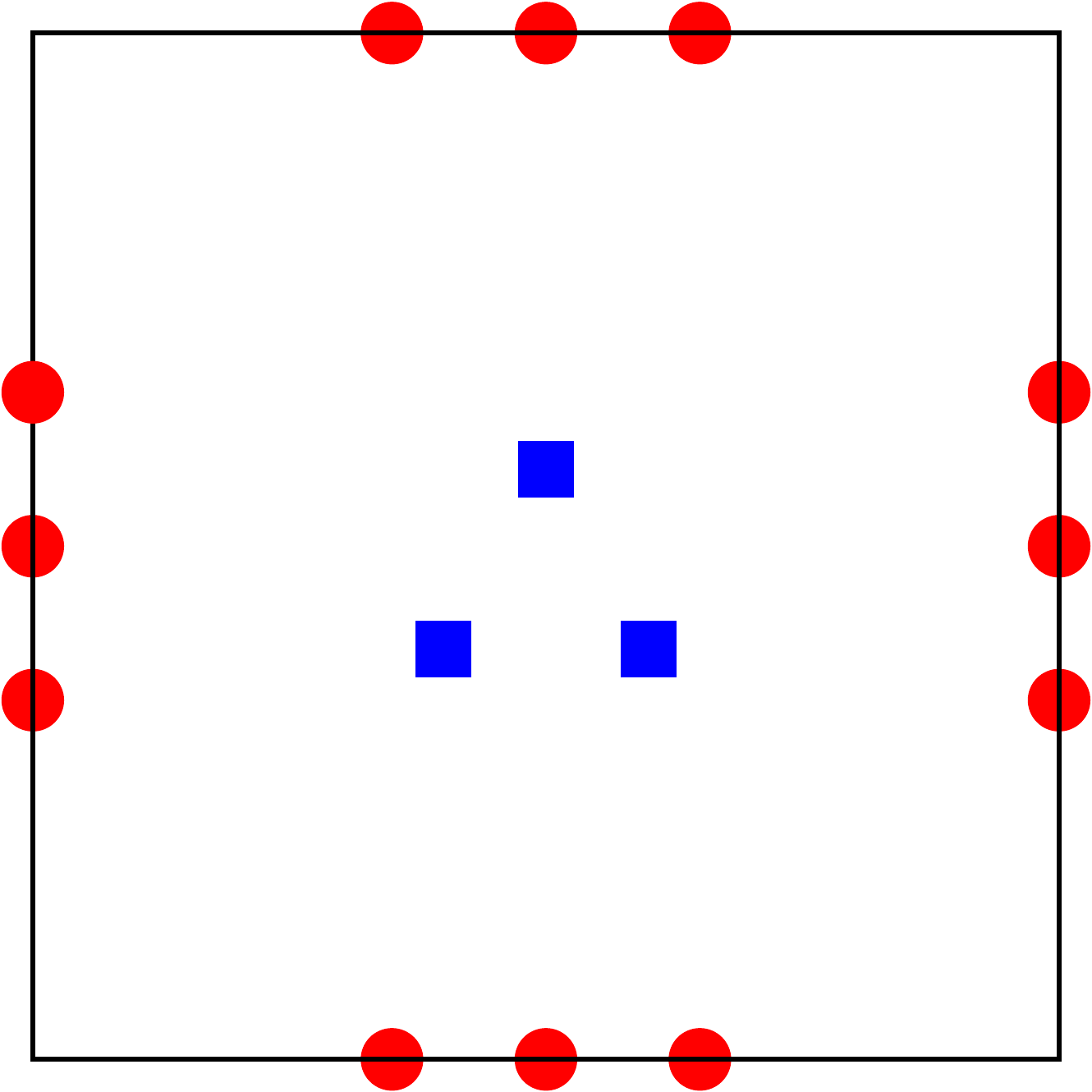} &\quad
    \includegraphics[width=0.2\textwidth]{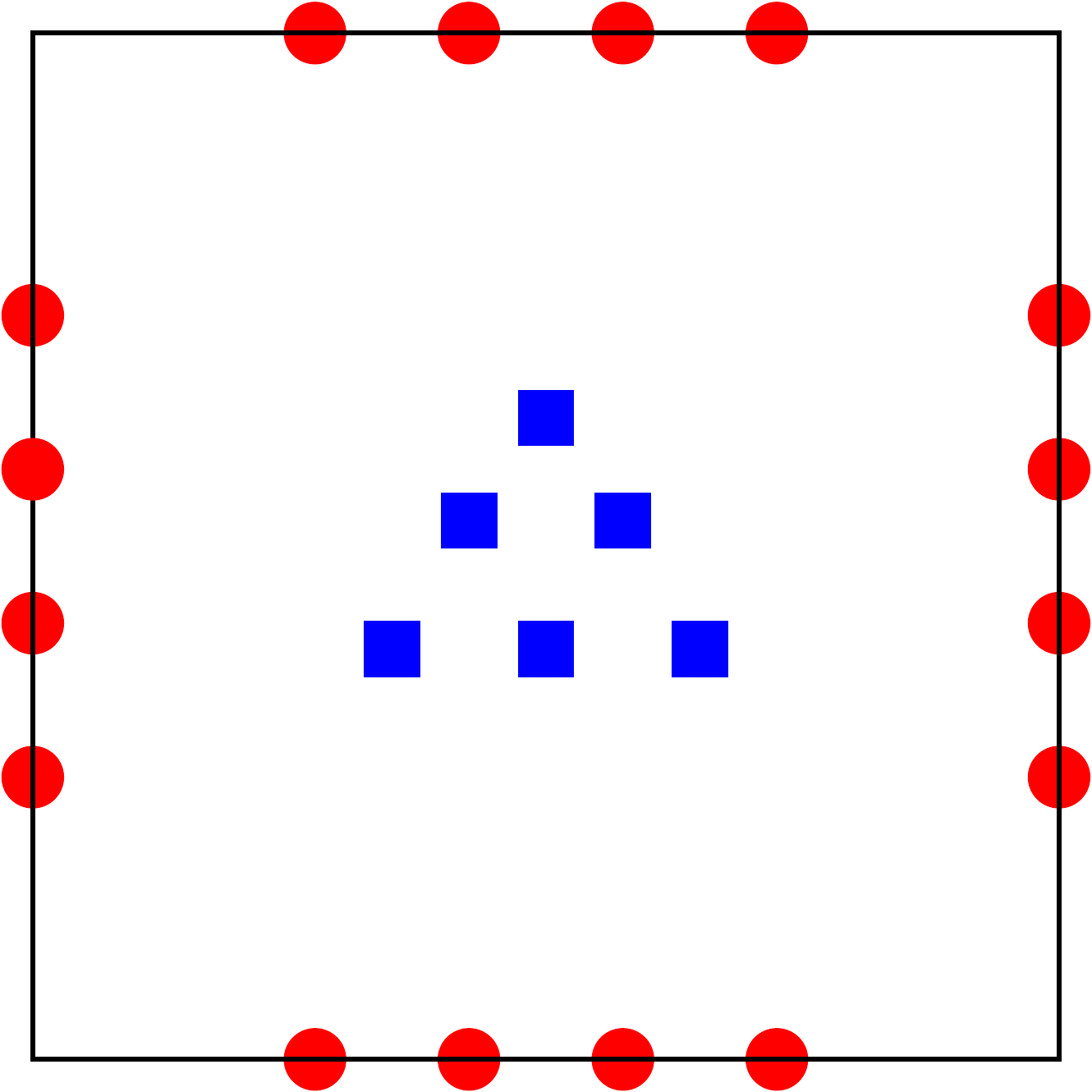} \\
    $\mathbf{k=1}$ & $\mathbf{k=2}$ & $\mathbf{k=3}$ & $\mathbf{k=4}$
  \end{tabular}
  \caption{Degrees of freedom of a quadrilateral cell for $k=1,2,3,4$; edge
    moments are marked by a circle; cell moments are marked by a square.}
  \label{fig:dofs:quadrilateral}
\end{figure}

\begin{figure}[!t]
  \centering
  \begin{tabular}{cccc}   
    \includegraphics[width=0.2\textwidth]{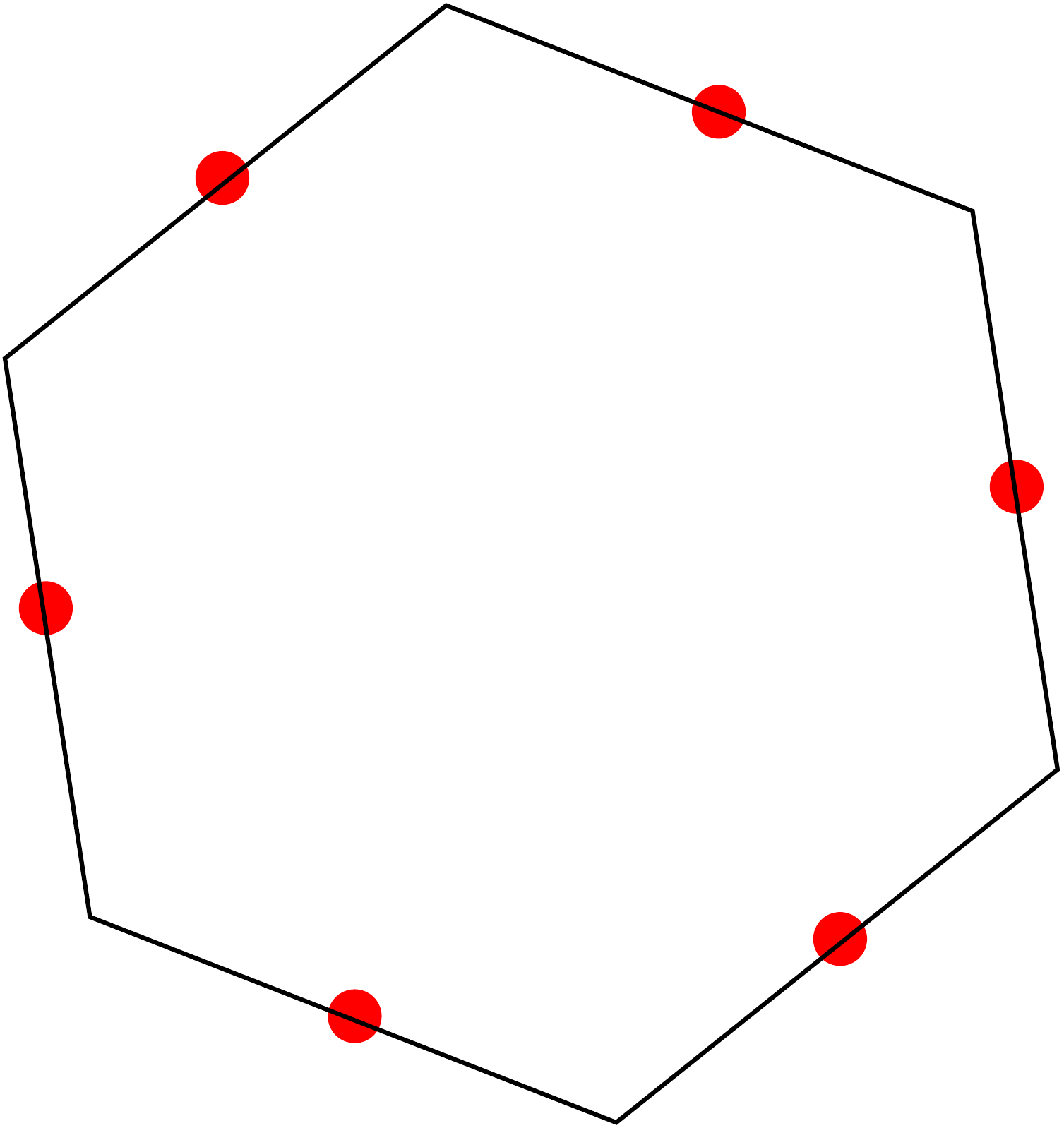} &\quad
    \includegraphics[width=0.2\textwidth]{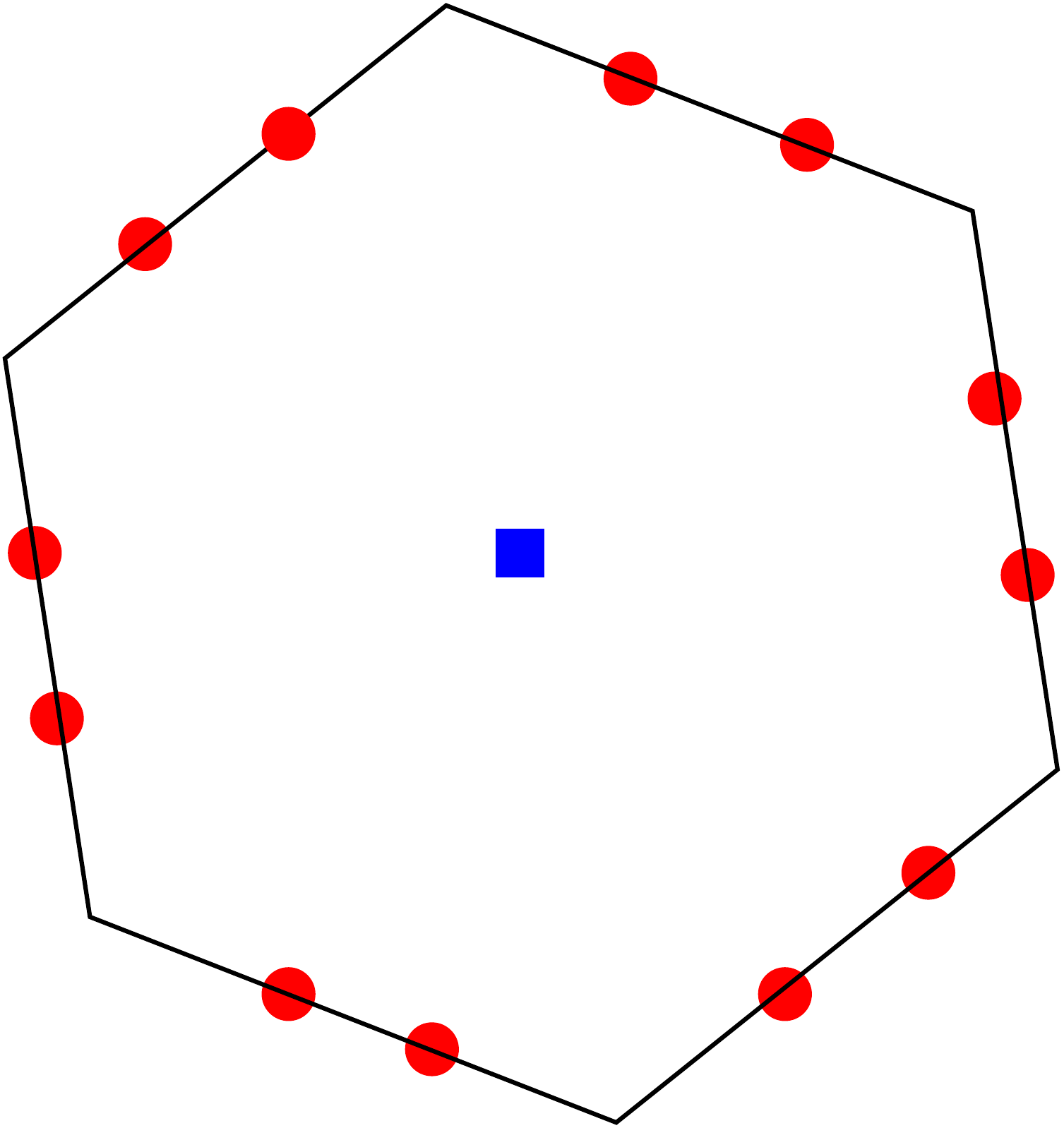} &\quad
    \includegraphics[width=0.2\textwidth]{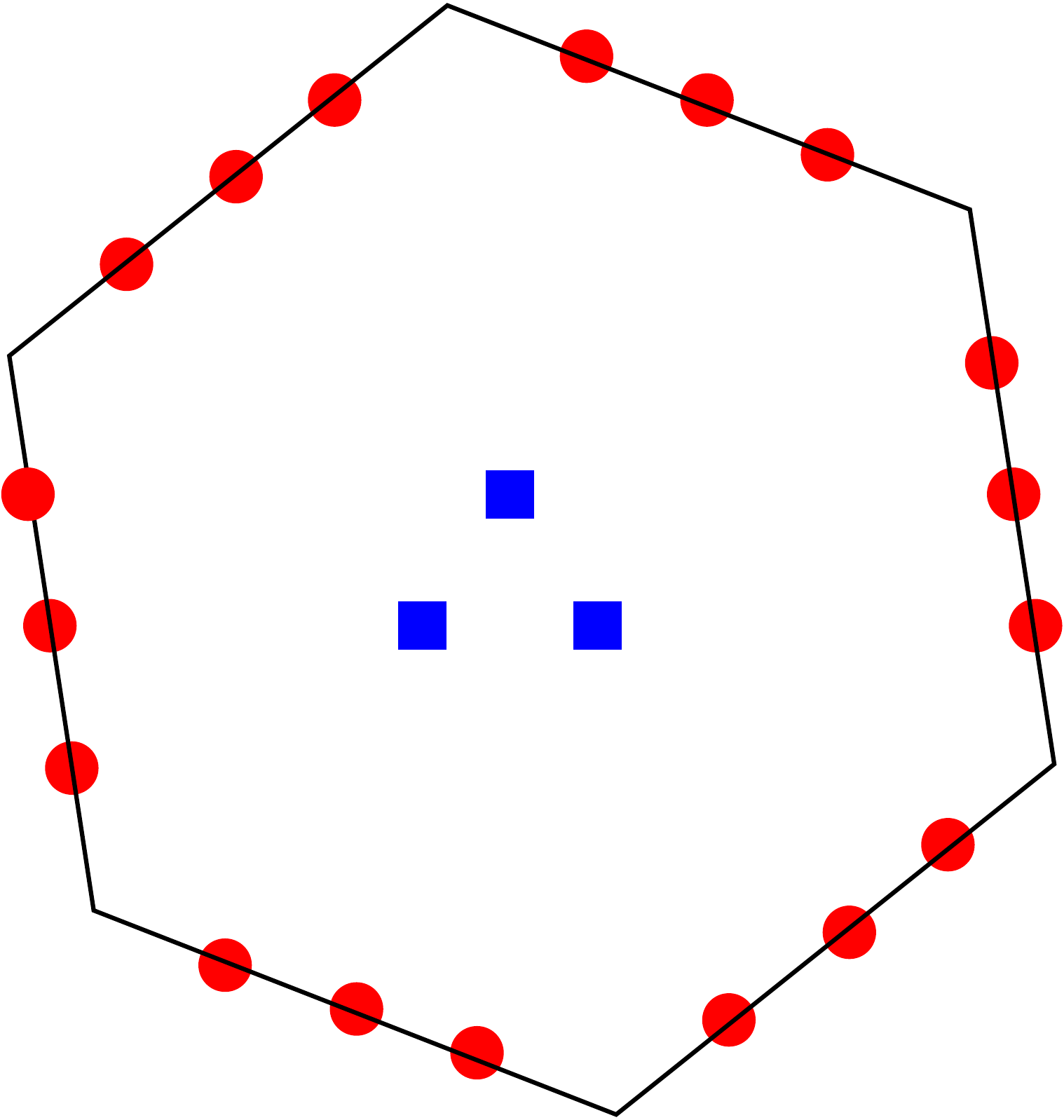} &\quad
    \includegraphics[width=0.2\textwidth]{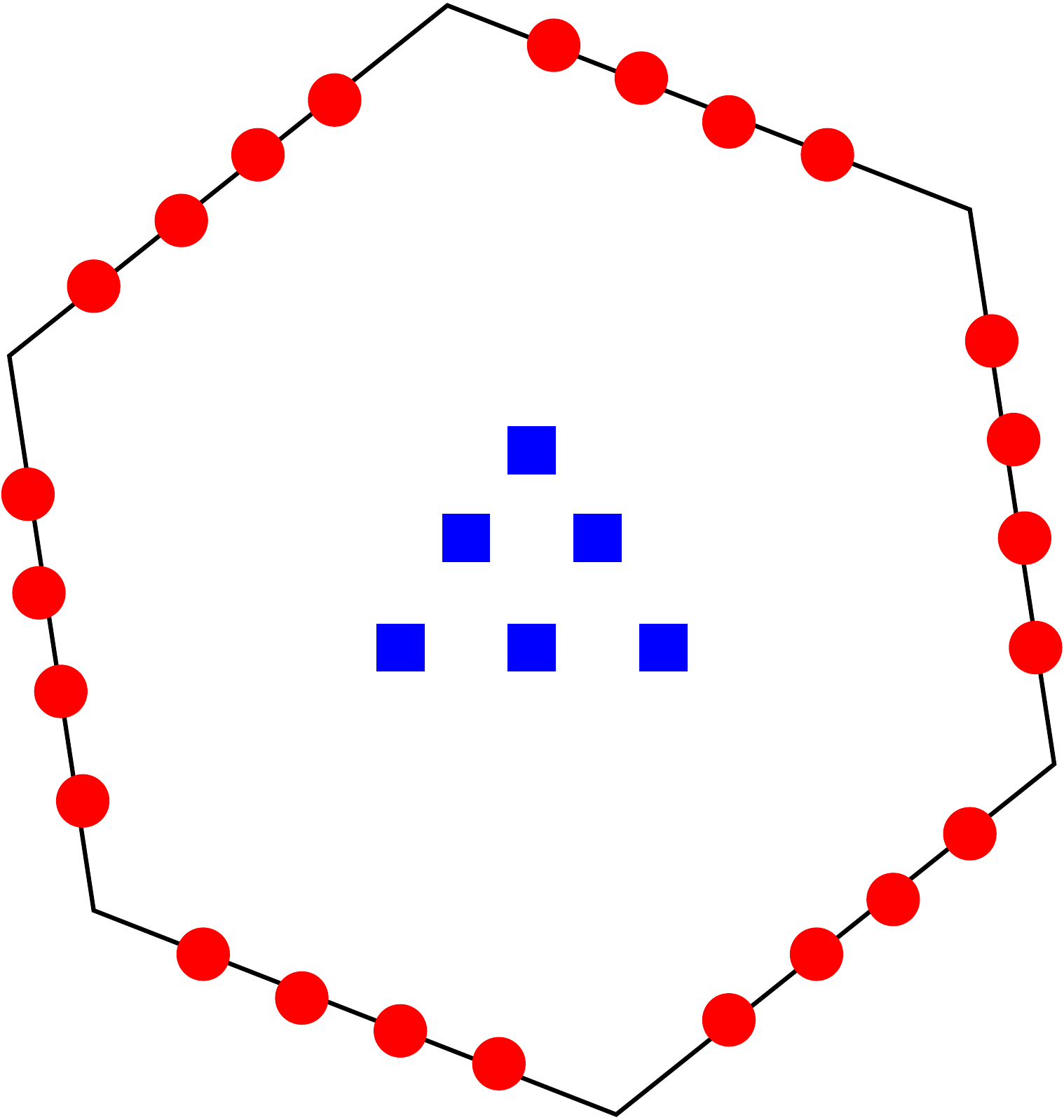} \\
    $\mathbf{k=1}$ & $\mathbf{k=2}$ & $\mathbf{k=3}$ & $\mathbf{k=4}$
  \end{tabular}
  \caption{Degrees of freedom of a hexagonal cell for $k=1,2,3,4$; edge
    moments are marked by a circle; cell moments are marked by a square.}
  \label{fig:dofs:hexagon}
\end{figure}

\begin{figure}[!t]
  \centering
  \begin{tabular}{cccc}   
    \includegraphics[width=0.2\textwidth]{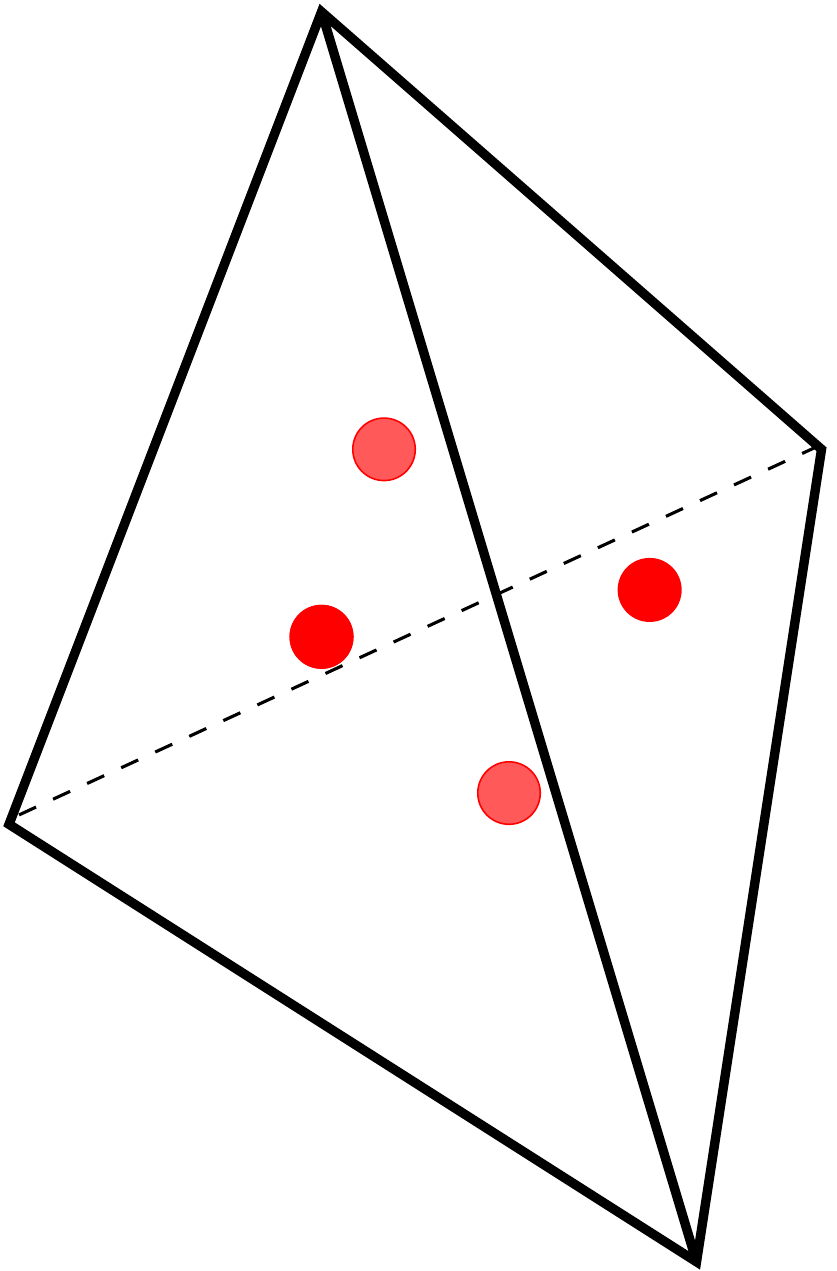} &\quad
    \includegraphics[width=0.2\textwidth]{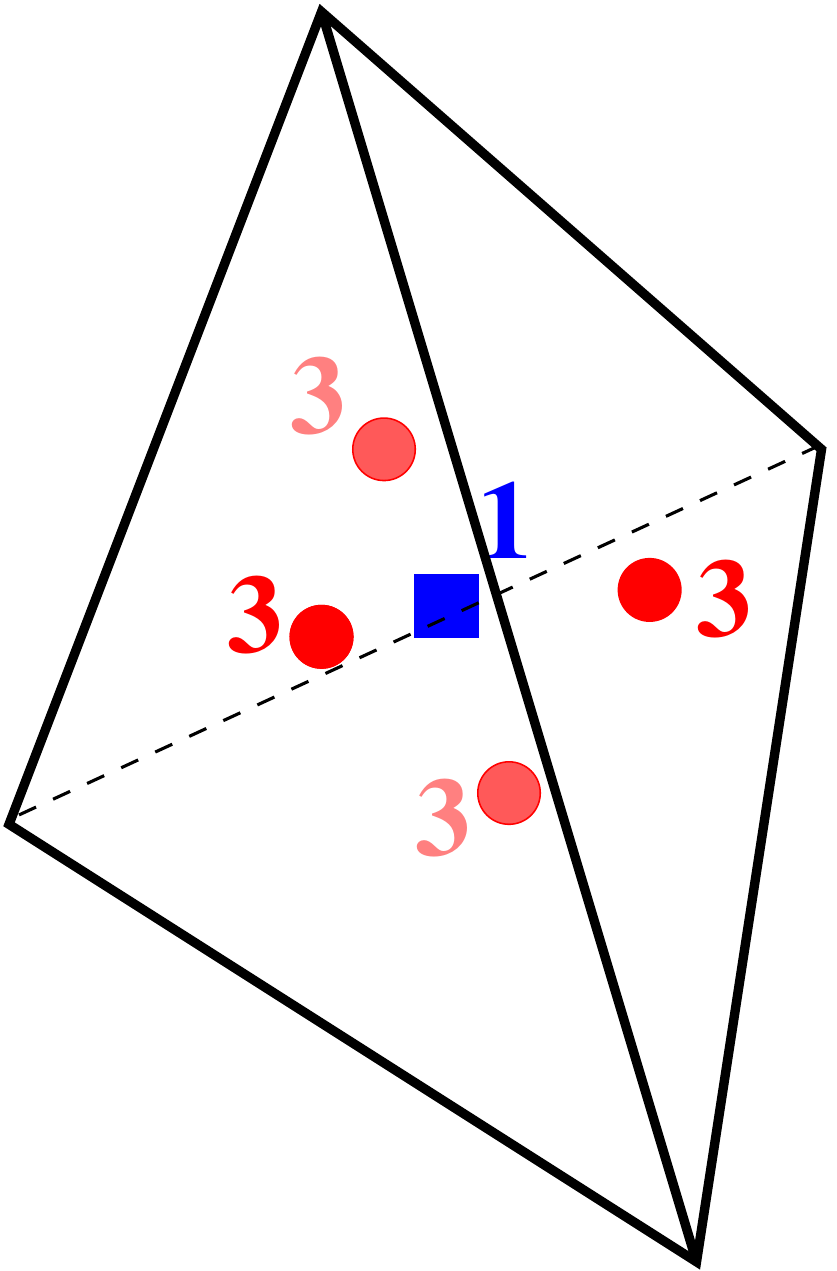} &\quad
    \includegraphics[width=0.2\textwidth]{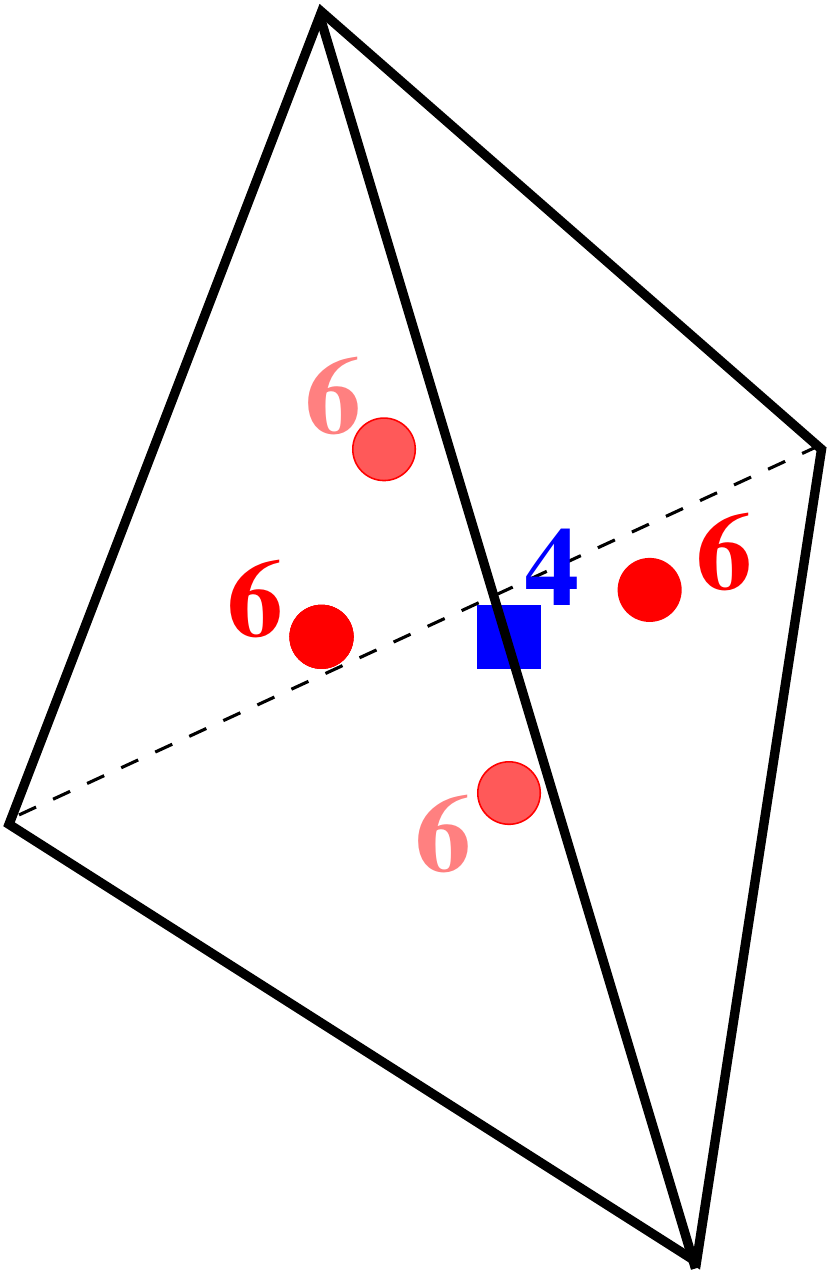} &\quad
    \includegraphics[width=0.2\textwidth]{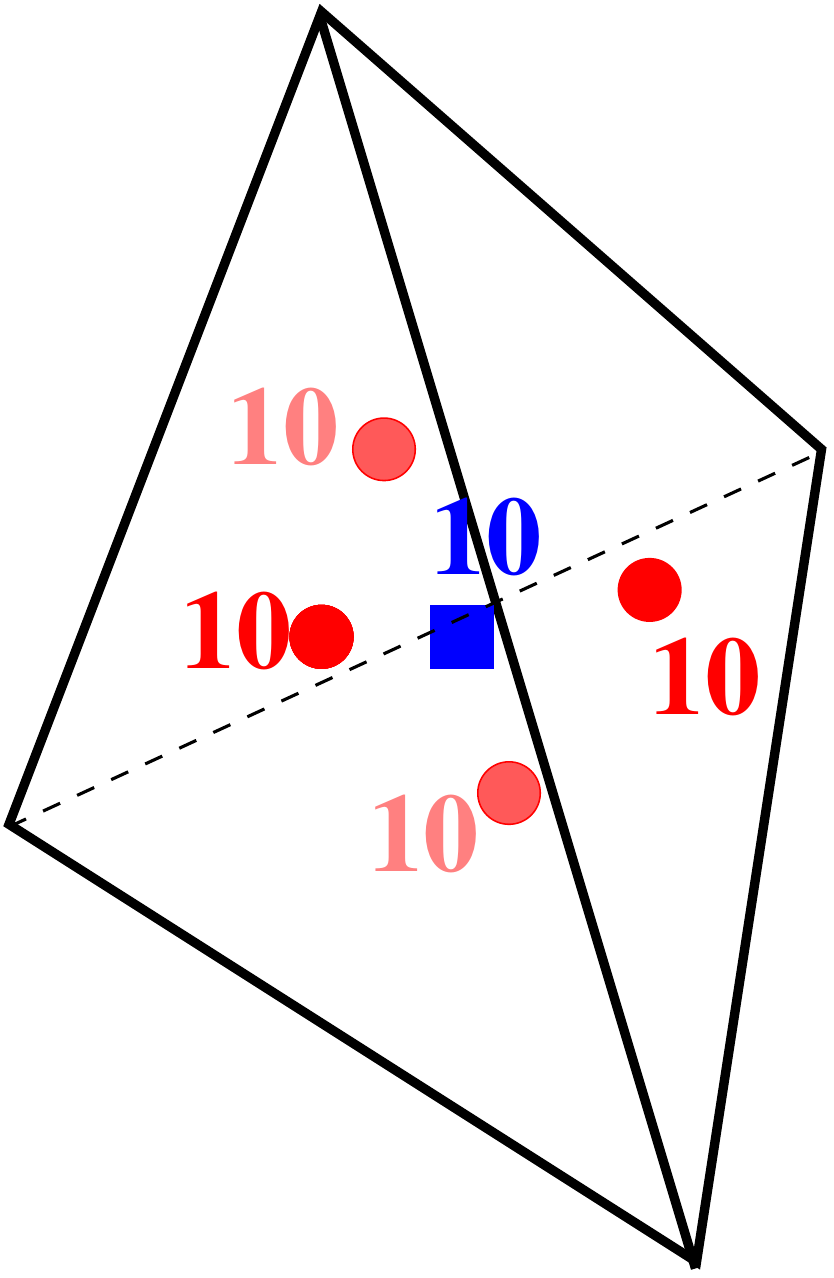} \\
    $\mathbf{k=1}$ & $\mathbf{k=2}$ & $\mathbf{k=3}$ & $\mathbf{k=4}$
  \end{tabular}
  \caption{Degrees of freedom of a tetrahedral cell for $k=1,2,3,4$; face
    moments are marked by a circle; cell moments are marked by a square.
    The numbers indicates the number of degrees of freedom ($1$ is not marked in the plot
    for $k=1$).}
  \label{fig:dofs:tetbis}
\end{figure}

\begin{figure}[!t]
  \centering
  \begin{tabular}{cccc}   
    \includegraphics[width=0.2\textwidth]{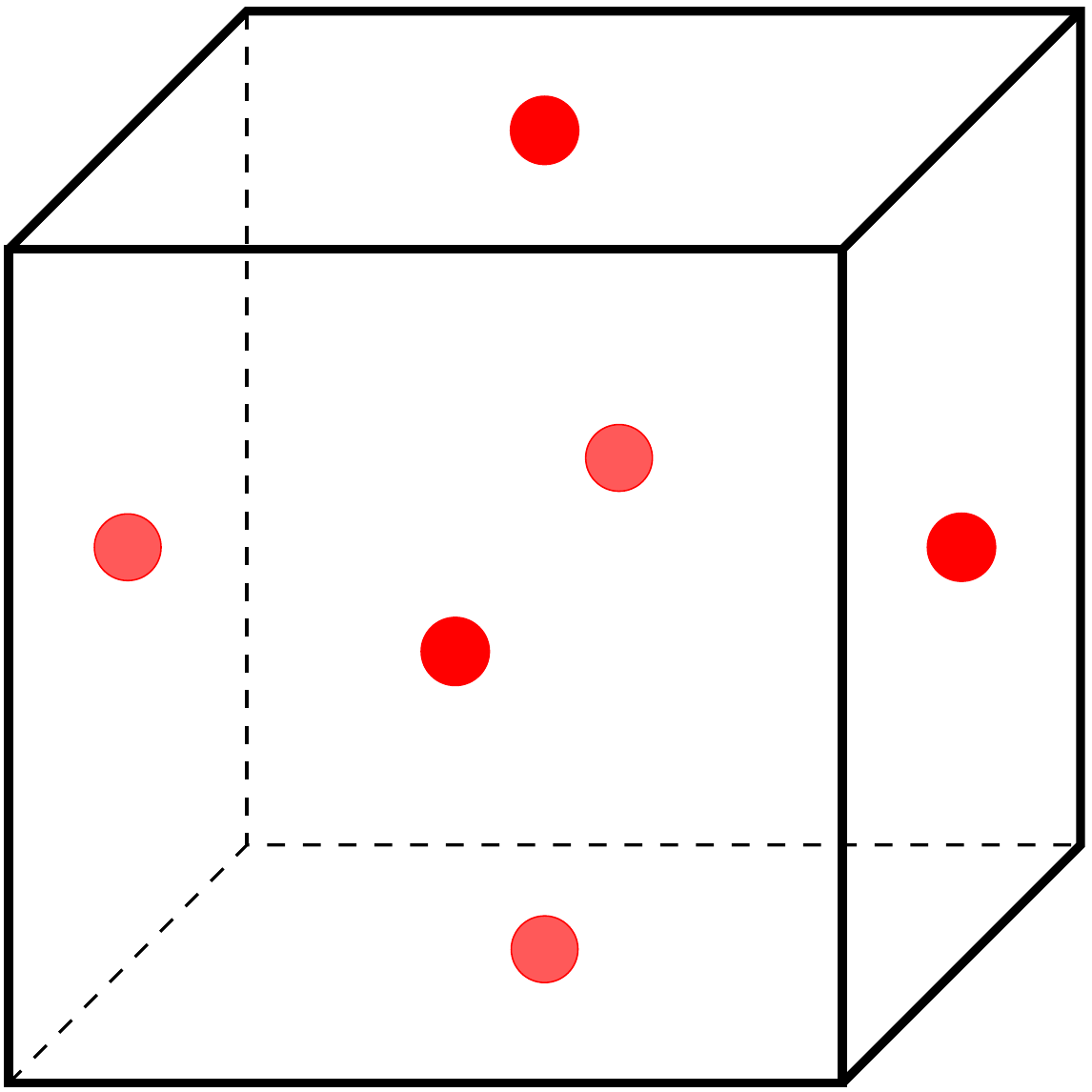} &\quad
    \includegraphics[width=0.2\textwidth]{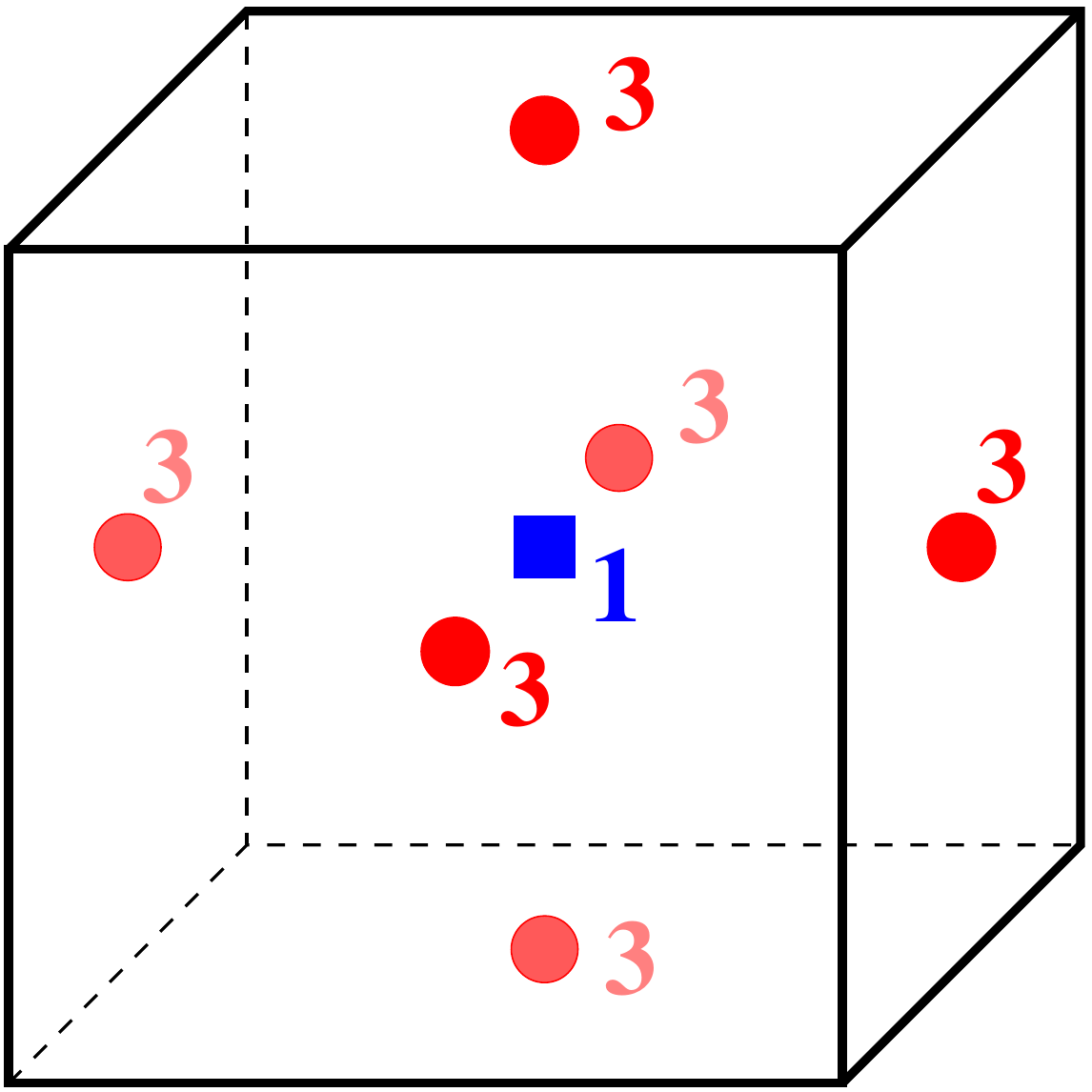} &\quad
    \includegraphics[width=0.2\textwidth]{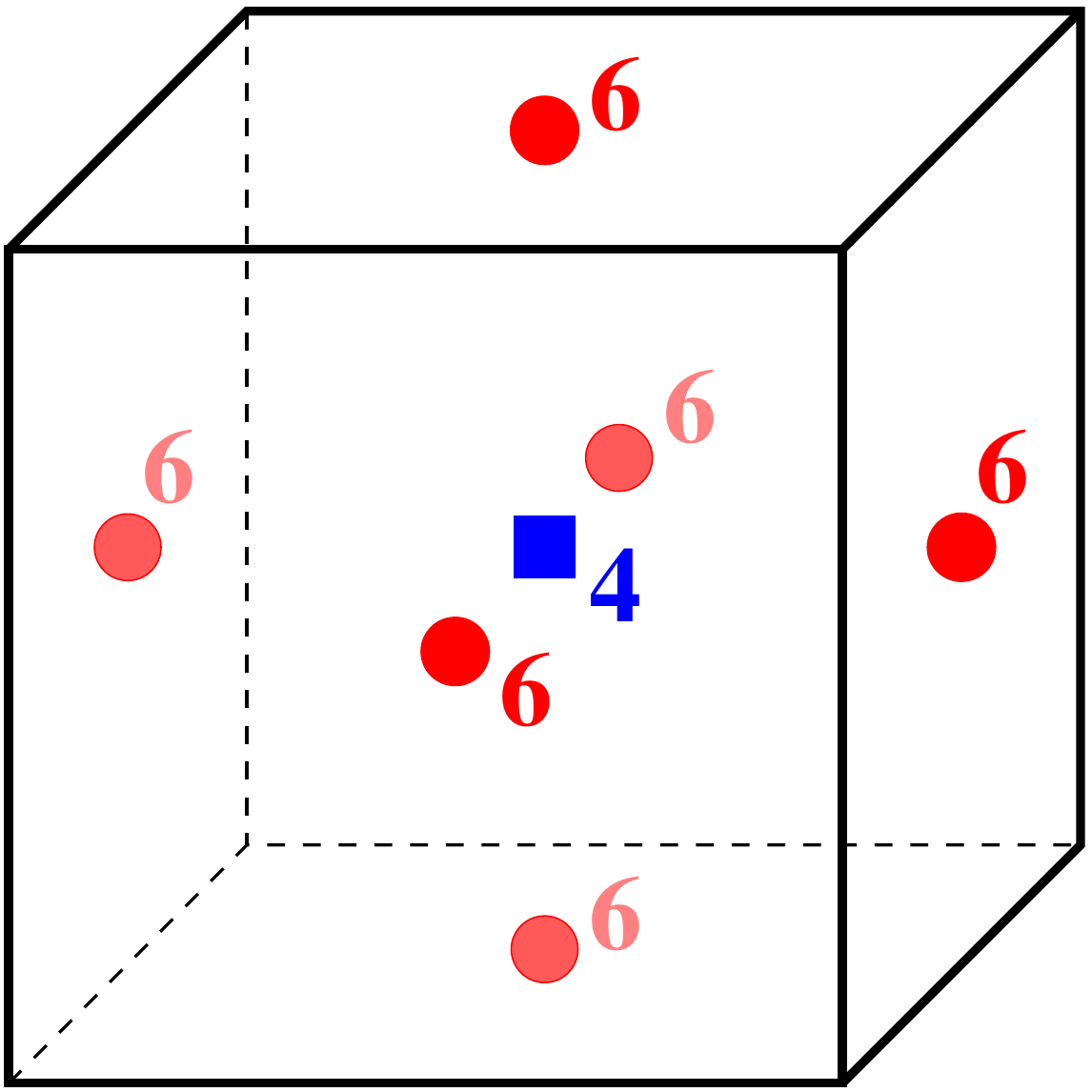} &\quad
    \includegraphics[width=0.2\textwidth]{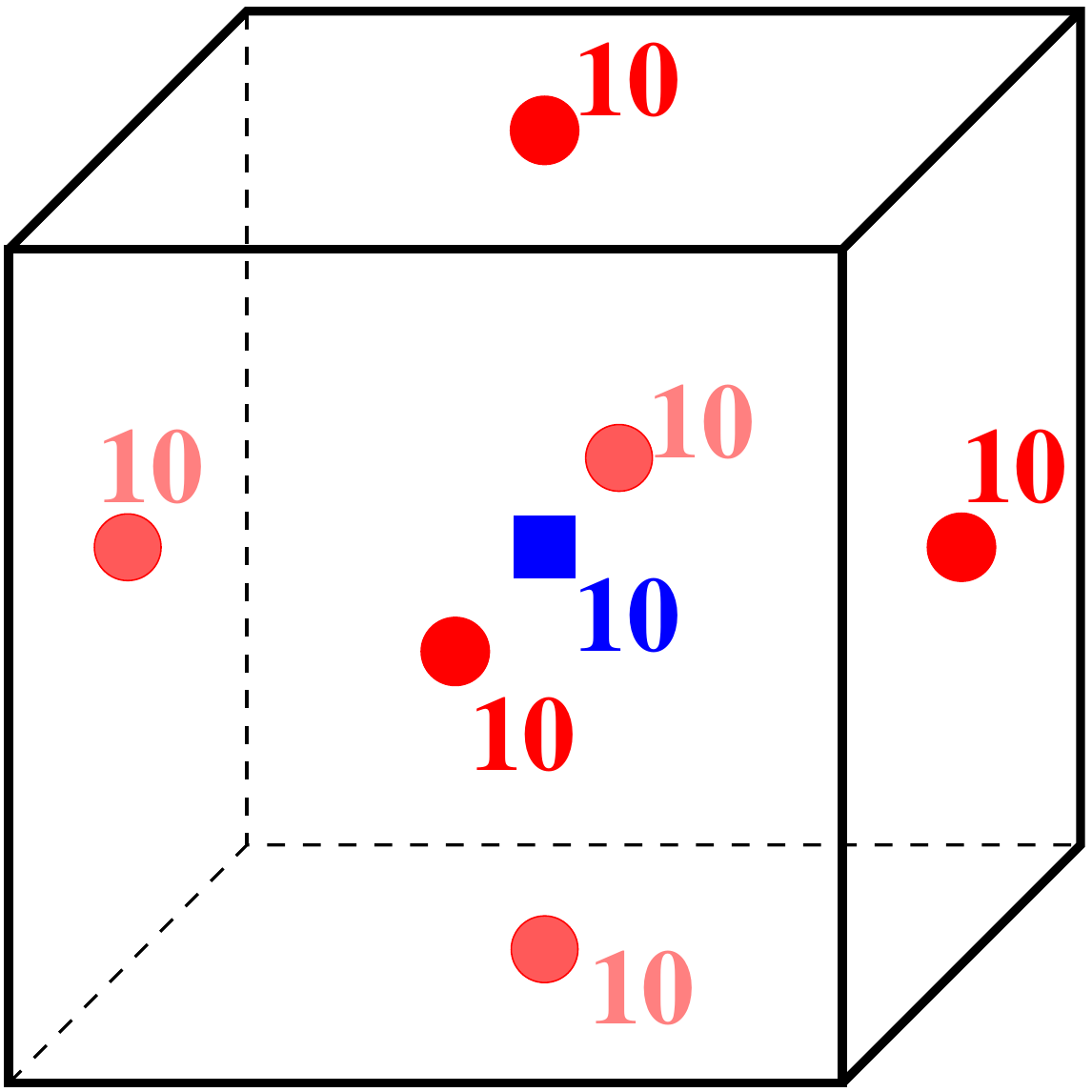} \\
    $\mathbf{k=1}$ & $\mathbf{k=2}$ & $\mathbf{k=3}$ & $\mathbf{k=4}$
  \end{tabular}
  \caption{Degrees of freedom of a cubic cell for $k=1,2,3,4$; face
    moments are marked by a circle; cell moments are marked by a square. 
    The numbers indicate the number of degrees of freedom ($1$ is not marked in the plot
    for $k=1$).}
  \label{fig:dofs:cube}
\end{figure}


In $V_h^{k}(\K)$ we can choose the following degrees of freedom:
\begin{description}
\item[$(i)$] all the moments of $\vh$ of order up to $\kk-1$ on each
  edge/face $\E\in\partial\K$:
  \begin{equation}\label{eq:dofs:01}
    \momE{\kk-1}(\vh)=\left\{\,
      \frac{1}{\mE}\int_{\E}\vh\,\mm\dS,
      \quad\forall\mm\in\calM^{\kk-1}(\E)
      \,\right\}
    \,\,
    \forall\E\subset\partial\K;
  \end{equation}
\item[$(ii)$] all the moments of $\vh$ of order up to $\kk-2$ on $\K$:
  \begin{equation}\label{eq:dofs:02}
    \momK{\kk-2}(\vh)=\left\{\frac{1}{\mK}\int_{\K}\vh\,\mm\dx,
    \quad\forall\mm\in\calM^{\kk-2}(\K)\right\}.
  \end{equation}
\end{description}
For $k=1,2,3,4$, the degrees of freedom are shown in 2D for a
triangular, a quadrilateral and an hexagonal element in
Figs.~\ref{fig:dofs:triangle}-\ref{fig:dofs:hexagon}, and in 3D for a
tetrahedral and a cubic element in
Figs.\ref{fig:dofs:tetbis}-\ref{fig:dofs:cube}.

Observe that the dimension $\sizeVhk$ given by \eqref{eq:Nk} coincides with the
total number of degrees of freedom defined in
\eqref{eq:dofs:01}-\eqref{eq:dofs:02}.
They are indeed unisolvent for the local space  $\Vhk(\K)$ as we show next:

\begin{lemma}\label{lemma:uni0}
  Let $\K$ be a simple polygon/polyhedra with $n$ edges/faces, and let
  $\Vhk(\K)$ be the space defined in \eqref{eq:def:Vhk0} for any
  integer $\kk\geq 1$.
  The degrees of freedom \eqref{eq:dofs:01}-\eqref{eq:dofs:02} are
  unisolvent for $\Vhk(\K)$.
\end{lemma}
\begin{proof}
Notice, that we cannot proceed as for the unisolvence proofs in classical finite elements, since  $V_h^{k}(\K)$ would contain typically functions that are not polynomial. Still we need to show that for any $\vh \in \Vhk(\K)$ such that
\begin{equation}\label{dofs:vanish}
\momE{\kk-1}(\vh)=0 \quad \forall\, \E\subset \partial\K \qquad \mbox{    and    } \qquad \momK{\kk-2}(\vh)=0
\end{equation}
then $v_h\equiv 0$. To do so, we use Divergence Theorem (with 
 $v_h\in V_h^{k}(\K)$  and so   $\frac{\partial\vh}{\partial\n} \,\in \,\mathbb{P}^{k-1}(e)$  on each $e\subset \partial \K$ and $ \Delta v\,\in\, \mathbb{P}^{k-2}(\K)$) to get
\begin{equation}
\int_{\K}\, | \nabla v_h|^{2} dx= -\int_{\K} v_h\, \Delta  \,v_h d\bx +\sum_{e\in \partial\K}\int_{e} v_h \, \frac{\partial\vh}{\partial\n}ds =0\;,
\end{equation}
where we have set the right hand side equal to zero using the fact that the degrees of freedom of $v_h$ vanish \eqref{dofs:vanish}. Hence $\nabla v_h=0 $ in $\K$ and so $v_h =$ constant in $\K$. But, since $\momE{0}(\vh)=0$ (the zero-order moment on each $e\subset \partial \K$ vanish), we deduce  $v_h\equiv 0$ in $\K$. 
\end{proof}

\begin{remark}
  The degrees of freedom of the method~\eqref{eq:dofs:01}-\eqref{eq:dofs:02} are defined by using the monomials in $\calM^{\kk-1}(\E)$ and $\calM^{\kk-2}(\K)$ as basis functions for the polynomial spaces $\Pp^{\kk-1}(\E)$ and
  $\Pp^{\kk-2}(\K)$.
  This special choice of the basis functions gives the method an
  inherent hierarchical structure with respect to $\kk$, which may be
  useful for an efficient implementation. However, the construction of the element
  is independent of such
  choice and, in principle, any other basis (properly defined and scaled) could be used to define the degrees of freedom.
\end{remark}

\subsection{The global nonconforming virtual element space $\Vhk$}
We now introduce the nonconforming  (global) virtual element space $V_h^{k}$ of order $k$. 
For every decomposition $\Th$ into elements $\K$ (polygons or polyhedra)  and  for every $\K\in\Th$, we consider the local space $\Vhk(\K)$ with
$\kk\ge1$ as defined in~\eqref{eq:def:Vhk0}.
Then, the \emph{global} nonconforming virtual element space $\Vhk$ of order
$\kk$ is given by
\begin{align}
  \Vhk=\big\{\,\vh\in\Hnc\,:\,\vh|_{\K}\in\Vhk(\K)\quad\forall\K\in\Th\,\big\}.
  \label{eq:def:nvem0}
\end{align}

Arguing then as for a single element, we can  compute  the degrees of freedom of the global space:
\begin{align}\label{eq:dim:tot}
  \sizeVhkg =
  \begin{cases}
    \NE\,\kk          + \NK(\kk-1)\kk/2        & \textrm{for~} d=2,\\[0.25em]
    \NF\,\kk(\kk+1)/2 + \NK(\kk-1)\kk(\kk+1)/6 & \textrm{for~} d=3.
  \end{cases}
\end{align}
where $N_{\text{elements}}$ denotes the total number of elements $\K$ of the partition $\Th$ and  $N_{\text{edges}}$ and $N_{\text{faces}} $ refer to the total number of edges (in $d=2$) and faces (in $d=3$), respectively; they are the cardinality of the set $\Eh$.

Arguing again as for a single element, as degrees of freedom for the global space $V_h^{k}$ we can take:

\begin{description}
\item[$(i)$] the moments of $\vh$ of order up to $\kk\!-\!1$ on each
  $(d\!-\!1)$-dimensional edge/face $\E\in\Eh$:
  \begin{equation}\label{eq:dofs:01g}
    \momE{\kk-1}(\vh)=\left\{\,
      \frac{1}{\mE}\int_{\E}\vh\,\mm\dS,
      \quad\forall\mm\in\calM^{\kk-1}(\E),
      \,\,\forall\E\in\Eh
      \,\right\};
  \end{equation}
\item[$(ii)$] the moments of $\vh$ of order up to $\kk\!-\!2$ on each
  $d$-dimensional element $\K\in\Th$:
  \begin{equation}\label{eq:dofs:02g}
    \momK{\kk-2}(\vh)=\left\{\,
      \frac{1}{\mK}\int_{\K}\vh\,\mm\dx,
      \quad\forall\mm\in\calM^{\kk-2}(\K),
      \,\,\forall\K\in\Th
      \,\right\}.
  \end{equation}
\end{description}
As it happens for the local space $V_h^{k}(\K)$, the dimension
$N^{\text{tot}}$ given in \eqref{eq:dim:tot} coincides with the total
number of degrees of freedom
\eqref{eq:dofs:01g}-\eqref{eq:dofs:02g}. The unisolvence for the local
space $V_h^{k}(\K)$ given in Lemma~\ref{lemma:uni0}, implies simply
the unisolvence for the global space $V_h^{k}$. Since the proof is
essentially the same, we omit it for conciseness.

\subsection{Approximation properties}
Following \cite{volley}, we now revise the local approximation properties  by  polynomial functions and functions in the virtual nonconforming space. In the former case, the approximation is  obviously exactly the same as for standard and classical finite elements. In the latter case, the discussion is similar  as for conforming  \VEM{}. We briefly  recall both  for completeness of
exposition and future reference in the paper.

\PAR{Local approximation}
In view of the mesh regularity assumptions \ASSUM{A}{1}-\ASSUM{A}{3},
there exists a local polynomial approximation
$\wp\in\Pp^{\kk}(\K)$ for every smooth function $w$ defined on
$\K$.
According to~\cite{brenner-scott} for star-shaped elements  and the
generalization to the general shaped elements 
satisfying~{\bf (A0)} found
in~\cite[Section~1.6]{BeiraodaVeiga-Lipnikov-Manzini:2013:book}, the
polynomial $\wp$ has optimal approximation properties.
Thus, for every $w\in H^{s}(\K)$ with $2\leq s\leq\kk+1$ there exists
a polynomial $\wp$ in $\Pp^{\kk}(\K)$ such that
\begin{equation}\label{eq:approx:pi}
  \norm{w-\wp}_{0,\K} + \hK\snorm{w-\wp}_{1,\K}\leq C\hK^{s}\snorm{w}_{s,\K} 
  \quad w\in H^{s}(\K),\,\,2\leq s\leq\kk+1,
\end{equation} 
where $C$ is a positive constant that only depends on the
polynomial degree $\kk$ and the mesh regularity constant $\varrho$.

\PAR{Interpolation error}
Following essentially  \cite{volley,proiettori} we can define an interpolation operator in $V_h^{k}$ having optimal approximation properties. The idea is to use the dofs but without requiring an explicit construction of basis functions for $V_h^{k}$ associated to those dofs, since, unlike for classical fem, this is not needed for implementing or constructing the method. We assume that we have numbered the degrees of freedom \eqref{eq:dofs:01g}-\eqref{eq:dofs:02g} from $i=1,\ldots\sizeVhkg$, and that  we have the canonical basis associated or induced by them  (even if we do not compute such basis!). We denote by $\chi_i$ the operator that to each smooth enough function $\phi$ associates the $i-th$ degree of freedom
\begin{equation}\label{eq:can00}
  v\,\,\to\,\,\chi_i(v)=\textrm{i-th degree of freedom of~}v\qquad\forall i=1,\ldots\sizeVhkg,
\end{equation}
and the ``canonical'' virtual basis functions $\psi_i$ of $\Vhk$
satisfying the condition $\chi_i(\psi_j)=\delta_{ij}$ for
$i,j=1,\ldots,\sizeVhkg$.
Then, from the previous construction of the space, it follows easily that for any $v\in \Hnc$, there exists some $v_{I}\in V_h^{k}$  such that
\begin{equation}\label{interp:0}
\chi_{i}(v-\vI)=0 \qquad \forall\,i=1,\ldots\sizeVhkg\;.
\end{equation}
All this is enough to guarantee that we can apply the classical results of approximation. In particular, we have that
there exists a constant $C>0$, independent of $\hh$ such that for every $h>0$, every $\K\in \Th$, every $s$ with $2\leq s \leq k+1$ and every $v \in H^{s}(\K)$ the interpolant $v_{I}\in V_h^{k}$ given in  \eqref{interp:0} satisfies:
\begin{equation}\label{eq:interp:00}
  \norm{v-\vI}_{0,\K} + \hK\snorm{v-\vI}_{1,\K}\leq C\hK^{s}\snorm{v}_{s,\K }.
\end{equation}
The proof of the above approximation result can be done proceeding as for classical finite elements (see for instance \cite{verfurth-bh} which can be used taking also into account \eqref{eq:poinc:02}).

\subsection{Construction of $a_h$}
\label{subsec:ah:construction}
We now tackle the second part of the definition of the nonconforming
virtual discretization \eqref{eq:var0h}. The goal is to define a
suitable symmetric discrete bilinear form $a_h : \Vhk \times\Vhk \lor
\mathbb{R}$ enjoying good stability and approximation properties, and
ensuring at the same time that the defined bilinear form
$a_h(\cdot,\cdot)$ is indeed computable over functions in $\Vhk$. We
first split $a_h(\cdot,\cdot)$ as we did for $a(\cdot,\cdot)$ in
\eqref{eq:split0}:
\begin{align*}
  \ah{\uh}{\vh} = \sum_{\K\in\Th}\aKh{\uh}{\vh} \qquad\forall\uh,\vh\in\Vhk,
\end{align*}
with $a_h^{\K}: V_{h}^{k}(\K) \times V_{h}^{k}(\K) \lor \mathbb{R}$
denoting the restriction to the local space $V_{h}^{k}(\K)$. We now
look at the local construction of $a_h^{\K}(\cdot,\cdot)$.

We start by noticing that on each element $\K$, for
$p\in\Pp^{\kk}(\K)$ and $v_h\in V^{k}_h(\K)$ one can compute exactly
$a^{K}(p,v_h)$ by using only the local dofs given in
\eqref{eq:dofs:01}-\eqref{eq:dofs:02}. In fact, since
\begin{equation}\label{eq:obs:00}
  \aK{p}{\vh} = 
  \int_{\K}\nabla p\cdot\nabla\vh\dx = 
  -\int_{\K} \vh\,\Delta p \dx +\int_{\partial\K} \vh\,\frac{\partial p}{\partial\n}\dx\;,
\end{equation}
one only needs to observe that the two integrals on the right hand
side above are determined exactly from the dofs
\eqref{eq:dofs:01}-\eqref{eq:dofs:02}, without requiring any further
explict knowledge of the function $v_h$ in $\K$.

To construct now $a_h(\cdot,\cdot)$, always following
\cite{proiettori} we first define a projection operator that can be
thought of as the Ritz-Galerkin projection in the classical finite
elements.  Let $ \Pi^{\nabla} : H^{1}(\K) \lor \mathbb{P}^{k}(\K)$ be
defined by
\begin{equation}\label{eq:def:Pia}
  \int_{\K}\nabla(\PINK(\vh)-\vh)\,\nabla q\dx=0 \qquad \forall q\in\Pp^{\kk}(\K),\,\vh\in\Vhk(\K)
\end{equation}
together with the  condition
\begin{align}
  \int_{\partial\K}(\PINK(\vh)-\vh)\dS&=0 \qquad\textrm{if~}k=1,   \label{eq:def:Pib:1}\\[0.5em]
  \int_{\K}       (\PINK(\vh)-\vh)\dx&=0 \qquad\textrm{if~}k\ge 2.\label{eq:def:Pib:k}
\end{align}
Note that $ \Pi^{\nabla} (v)$ is indeed computable for any $v\in
V_h^{k}$ from the degrees of freedom
\eqref{eq:dofs:01}-\eqref{eq:dofs:02} in view of the observation
\eqref{eq:obs:00} and the symmetry of the bilinear form. Note also that
$\Pi^{\nabla}$ is the identity operator on $\Pp^{\kk}(\K)$, i.e.,
$\PINK\big(\Pp^{\kk}(\K)\big)=\Pp^{\kk}(\K)$.

We then define
\begin{equation}\label{eq:def:ah}
  \aKh{\uh}{\vh}=\aK{\PINK(\uh)}{\PINK(\vh)} + \SK{\uh-\PINK(\uh)}{\vh-\PINK(\vh)} \qquad \forall\, u,v\in V_h^{k}(\K)\;,
\end{equation}
where the term $S^{\K}(\cdot,\cdot)$ 
is a symmetric bilinear form whose matrix representation in the
canonical basis functions $\{\psi_i\}$ of $\Vhk(\K)$ is spectrally
equivalent to the identity matrix scaled by the factor $\gamma_{\K}$
defined as:
\begin{equation}\label{eq:def:gamma}
  \gamma_{\K}= \hh^{d-2}.
\end{equation}
Thus, for every function $\vh$ in $\Vhk(\K)$, it holds that
\begin{align}
  \label{eq:def:SK}
  \SK{\vh}{\vh}\simeq
  \hh^{d-2}\vect{v}_{\hh}^{t}\vect{v}_{\hh}
\end{align}
where $\vect{v}_{\hh}$ is the vector collecting the degrees of freedom
of $\vh$.
The scaling of $S^{\K}$ guarantees that
\begin{equation}\label{eq:per-stab}
  c^{\ast}\aK{\vh}{\vh}\leq\SK{\vh}{\vh}
  \leq c_{\ast}\aK{\vh}{\vh}\qquad\forall\vh\in\textrm{ker}(\PINK),
\end{equation}
for some positive constants $c_{\ast}$ and $c^{\ast}$ independent of
$\hh$.

We now show that  the  construction of  $a_{h}^{\K}(\cdot,\cdot)$  guarantees the usual  consistency and stability properties in VEM:

\begin{lemma}\label{le:1}
  For all $h>0$ and for all $\K\in \Th$, the bilinear form
  $a_{h}^{\K}(\cdot,\cdot)$ defined in \eqref{eq:def:ah} satisfies the
  following consistency (with respect to polynomials
  $\mathbb{P}^{k}(\K)$) and stability properties :
\begin{itemize}
\item \emph{$\kk$-Consistency:} 
  \begin{equation}\label{eq:consistency}
    \aKh{p}{\vh}=\aK{p}{\vh}\qquad\forall p\in\Pp^{\kk}(\K),\,\,\forall\vh\in\Vhk(\K).
  \end{equation}
\item \emph{Stability:} there exists two positive constants
  $\alpha^{\ast}$ and $\alpha_{\ast}$ independent of mesh size $\hh$ but depending on the shape regularity of the partition such that
  \begin{equation}\label{eq:stability}
    \alpha_{\ast}\aK{\vh}{\vh}\leq\aKh{\vh}{\vh}\leq
    \alpha^{\ast}\aK{\vh}{\vh}\quad\forall\vh\in\Vhk(\K).
  \end{equation}
\end{itemize}
\end{lemma}

\begin{proof}
The $\kk$-consistency property in~\eqref{eq:consistency} follows immediately
from definition \eqref{eq:def:ah} and the fact that $\PINK$ is the
identity operator on $\Pp^{\kk}(\K)$.
Since $\Pi^{\nabla} \mathbb{P}^{k}(\K)=\mathbb{P}^{k}(\K)$, it follows
that $\SK{p-\PINK(p)}{\vh-\PINK(\vh)}=0$, and so using the definition
of $\Pi^{\nabla}$ and the definition \eqref{eq:def:ah} we have for
every $p\in\Pp^{\kk}(\K)$ and every $\vh\in\Vhk(\K)$,
\begin{equation*}
  \aKh{p}{\vh}=\aK{\PINK(p)}{\PINK(\vh)}=\aK{\PINK(p)}{\vh}=\aK{p}{\vh}
\end{equation*}
 which gives
\eqref{eq:consistency} and proves the $k$-consistency property.

To show \eqref{eq:stability}, from the definition of
$a_{h}^{\K}(\cdot,\cdot)$ given in \eqref{eq:def:ah}, the symmetry and
\eqref{eq:per-stab}, we have
\begin{align*}
  \aKh{\vh}{\vh}
  &\leq \aK{\PINK(\vh)}{\PINK(\vh)} +c_{\ast}\aK{\vh-\PINK(\vh)}{\vh-\PINK(\vh)}\\
  &\leq \max(1,c_{\ast})\big( \aK{\PINK(\vh)}{\PINK(\vh)}+\aK{\vh-\PINK(\vh)}{\vh-\PINK(\vh)} \big)\\
  &= \alpha^{\ast}\aK{\vh}{\vh},
\end{align*}
and
\begin{align*}
  \aKh{\vh}{\vh}
  &\geq \aK{\PINK(\vh)}{\PINK(\vh)} +c^{\ast}\aK{\vh-\PINK(\vh)}{\vh-\PINK(\vh)}\\
  &\geq \min(1,c^{\ast})\big( \aK{\PINK(\vh)}{\PINK(\vh)}+\aK{\vh-\PINK(\vh)}{\vh-\PINK(\vh)} \big)\\
  &= \alpha_{\ast}\aK{\vh}{\vh},
\end{align*}
which shows \eqref{eq:stability} with $\alpha_{\ast}=\min(1,c^{\ast})$ and  $\alpha^{\ast}:=\max(1,c_{\ast})$ and concludes the proof.
\end{proof}

Cauchy-Schwarz inequality, together with \eqref{eq:stability} and the boundness of the local continuous
bilinear form give
\begin{align}
  a_h^{\K} (u_h,v_h)
  & \leq (a_h^{\K} (u_h,u_h))^{1/2} (a_h^{\K} (v_h,v_h))^{1/2}\leq \alpha^{\ast} (a^{\K} (u_h,u_h) )^{1/2}(a^{\K} (v_h,v_h) )^{1/2} \nonumber\\
  &=  \alpha^{\ast} \|\nabla u_h\|_{0,\K} \|\nabla v_h\|_{0,\K} \qquad \forall \,\, u_h\, v_h \in V_h^{k}(\K)\;,
  \label{eq:cont0}
\end{align}
which establishes the continuity of $a^{K}_h$.\\

\subsection{Construction of the right-hand-side}
\label{subsec:fh:construction}.
The forcing term is constructed in the same way as it is done for the conforming VEM. The idea is to use, whenever is possible, the degrees of freedom \eqref{eq:dofs:02g} to compute exactly $f_h$. Denoting by $\mathcal{P}^{\ell}_{\K}:L^{2}(\K) \lor \mathbb{P}^{\ell}(\K)$ the $L^{2}$-orthogonal projection onto the space $\mathbb{P}^{\ell}(\K)$ for $\ell\ge0$, we define $f_h$ at the element level by:
\begin{equation}\label{eq:def-f}
(f_h)|_{\K}:=\left\{ \begin{aligned}
&\mathcal{P}_{\K}^{0}(f) \qquad &&\mbox{ for  } k= 1\;,\\
&\mathcal{P}_{\K}^{k-2}(f) \qquad &&\mbox{ for  } k\ge 2\;.
\end{aligned}\right. \qquad \qquad \forall\, \K\in \Th\;.
\end{equation}
In the above definition for $k\geq 2$, the right hand side $\langle f_h,v_h\rangle$ is fully computable for functions in $V_h^{k}$ since:
\begin{equation*}
\langle f_h,v_h\rangle:=\sum_{\K} \int_{\K} \mathcal{P}_{\K}^{k-2}(f) v_h d\bx=\sum_{\K} \int_{\K} f\mathcal{P}_{\K}^{k-2}(v_h) d\bx\;.
\end{equation*}
which is readily available from \eqref{eq:dofs:02g}.

For $k=1$,  for each $\K\in \Th$ we first define:
\begin{equation}\label{v:tilda}
\tilde{v}_h|_{\K}:=\frac{1}{n}\sum_{e\in \partial\K} \frac{1}{|e|} \int_{e} v_h ds \,\, \approx \,\, \mathcal{P}_{\K}^{0}(v_h)\;,
\end{equation}
and notice that $\tilde{v}_h|_{\K}$ is a first order approximation to $\mathcal{P}_{\K}^{0}(v_h)=\frac{1}{|\K|}\int_{\K} v_h dx$; i.e., $|\tilde{v}_h|_{\K}-\mathcal{P}_{\K}^{0}(v_h)|\leq Ch |v|_{1,\K}$. 
Then, the idea is to use $\tilde{v}_h$ to compute the approximation of the right hand side:
\begin{equation*}
\begin{aligned}
\langle f_h,\tilde{v}_h\rangle &:=\sum_{\K} \int_{\K} \mathcal{P}_{\K}^{0}(f) \tilde{v}_h d\bx \,\, \approx \,\, \sum_{\K} |\K|  \mathcal{P}_{\K}^{0}(f) \mathcal{P}_{\K}^{0}(v_h)\;. &&
\end{aligned}
\end{equation*}
Notice that the computation of the right most term above would require the knowledge of the average values of $v_h$ on each element $\K$ and such information , in priciple is not available.  Therefore, we approximate  $\mathcal{P}_{\K}^{0}(v_h)$ by using  the  numerical quadrature rule defined by $\tilde{v}_h$ that  only uses the moments $  \momE{0}(\vh)$ defined in \eqref{eq:dofs:01}.\\

Furthermore, in both cases ($k\geq 2$ and $k=1$), an estimate for the error in the approximation is already available by using the definition of the $L^{2}$-projection, Cauchy-Schwarz and standard approximation estimates \cite{ciar2}. For $ k\ge 2$ and $s\geq 1$  one easily has
\begin{align*}
\left| \langle f,v_h\rangle-\langle f_h,v_h\rangle \right|&=\left|  \sum_{\K} \int_{\K} (f-\mathcal{P}_{\K}^{k-2}(f)) v_h d\bx\right| &&\\
&=\left|  \sum_{\K} \int_{\K} \, (f-\mathcal{P}_{\K}^{k-2}(f))   \,(v_h -\mathcal{P}_{\K}^{0}(v_h)) d\bx\right| &&\\
&\leq\left\|\, \, (f-\mathcal{P}_{\K}^{k-2}(f))\right\|_{0,\Th} \big\| \,\, (v_h -\mathcal{P}_{\K}^{0}(v_h))\big\|_{0,\Th} &&\\
&\leq C h^{\min{(k,s)}} |f |_{s-1,\Th} |v_h|_{1,h}\;. &&
\end{align*}
For $k=1$,  the definition of $f_h$ together with using repeteadly the definition of the $L^{2}$-projection, Cauchy-Schwarz inequality and standard approximation estimates, give
\begin{align*}
\left| \langle f,v_h\rangle-\langle f_h,\tilde{v}_h\rangle \right|&=\left|  \sum_{\K} \int_{\K}\left( f\,v_h - \mathcal{P}_{\K}^{0}(f) \tilde{v}_h\right) d\bx\right| &&\\
&\leq \left|  \sum_{\K} \int_{\K}\left( f - \mathcal{P}_{\K}^{0}(f)) v_h\right) d\bx\right| +\left|  \sum_{\K} \int_{\K}\mathcal{P}_{\K}^{0}(f)\left( v_h -\tilde{v}_h\right) d\bx\right| &&\\
&\leq \left|  \sum_{\K} \int_{\K}\left( \, f-\mathcal{P}_{\K}^{0}(f)\, \right) \left(v_h - \mathcal{P}_{\K}^{0}(v_h)\right)   d\bx\right|  +\|\mathcal{P}_{\K}^{0}(f)\|_{0,\Th} \|v_h-\tilde{v}_h\|_{0,\Th}&&\\
&\leq  \left\|  \, (\mathcal{P}_{\K}^{0}(f)-f)\right\|_{0,\Th}\, \|\, \,( \mathcal{P}_{\K}^{0}(v_h)-v_h)\|_{0,\Th} +Ch \|f\|_{0,\O} |v_h|_{1,h} &&\\
&\leq   C  h^{\min{(k,s)}} |f |_{s-1,\Th} |v_h|_{1,h} +Ch \|f\|_{0,\O} |v_h|_{1,h} \;.&&
\end{align*}

\subsection{Construction of the boundary term}
In the case of non-homogenous Dirichlet boundary conditions, we need to construct the corresponding boundary term.  We define $g_h:=\mathcal{P}^{k-1}_e(g)$ and observe that 
 in view of the degrees of freedom \eqref{eq:dofs:01g}, with such definition  the boundary  term will be fully computable. Indeed,
\begin{equation}
\int_{\Ehb} g_h v_h \dS:=\sum_{e\in\Ehb} \int_e \mathcal{P}^{k-1}_e(g)v_h \dS=\sum_{e\in\Ehb} \int_e g\mathcal{P}^{k-1}_e(v_h) \dS \quad \forall\, v_h \in V_h^{k}\;.
\end{equation}

\section{Error Analysis}\label{sec:3}
In this section we present the error analysis, in the energy and
$L^{2}$-norms, for the nonconforming virtual element approximation
\eqref{eq:var0h} to the model problem \eqref{eq:var00}.

We start by noticing that the nonconformity of our discrete approximation space  $V^{k}_h\subset \Hnc \nsubseteq H^{1}(\O)$ introduces a kind of {\it consistency error} in the approximation to the solution $u\in V$. In fact it should be noticed that using \eqref{eq:split0} together with standard integration by parts give
\begin{align}\label{eq:a:nonc0}
a(u,v) &= \sum_{\K\in \Th} \int_{\K} -( \Delta u) v dx+ \sum_{\K\in \Th} \int_{\partial\K} \frac{\partial u}{\partial \n_{\K}} v \dS  &&\nonumber\\
&= (f,v) +\mathcal{N}_h(u,v) \qquad \quad \forall\, v \in {H}^{1,nc}(\Th;1)\;.
\end{align}
For $u \in H^{s}(\O), \,\, s\geq 3/2$, the term $\mathcal{N}_h$ can be rewritten as (using \eqref{eq:def-jump})
\begin{equation}\label{eq:def:Nh}
\mathcal{N}_h(u,v)=\sum_{\K\in \Th} \int_{\partial\K} \frac{\partial u}{\partial \n_{\K}}  v \dS  =\sum_{e\in \Eh}\int_{e}  \nabla u \cdot \jump{v} \dS.
\end{equation}

The term $\mathcal{N}_h$ measures the extent to which the continuous solution $u$ fails to satisfy the {\it virtual element} formulation \eqref{eq:var0h}.
In that respect, it could be regarded as a {\it consistency error} although it should be noted that such inconsistency here (as for classical nonconforming FEM) is due to the fact that the test functions $v_h\in V_h \nsubseteq V$, and therefore an error arises when using the variational formulation of the continuous solution \eqref{eq:var00}. 

We now provide an estimate for the term measuring the nonconformity. We have the following result
\begin{lemma}\label{le:Nh}
Assume  {\bf (A0)} is satisfied. Let $k\ge1$ and let $u\in H^{s+1}(\O)$ with $s\geq 1$ be the solution of \eqref{eq:var00}.
Let  $v\in H^{1,nc}(\Th;1)$ as defined in \eqref{eq:h1:02b}. Then, there exists a constant $C>0$ depending only on the polynomial degree and the mesh regularity such that 
\begin{equation}
|\mathcal{N}(u,v)| \leq C  h^{\min(s,k)} \|u\|_{s+1,\O}|v|_{1,h}
\end{equation}
where  $\mathcal{N}_h(u,v)$ is defined in \eqref{eq:def:Nh}.
\end{lemma}

\begin{proof}
The proof follows along the same line as the one for classical nonconforming methods. We briefly report it here for the sake of simplicity.
From the definition of the space $H^{1,nc}(\Th;k)$ with $k=1$, the definition of the $L^{2}(e)$-projection and Cauchy-Schwarz we find 
\begin{align}\label{eq:estima:Nh}
|\mathcal{N}_h(u,v_h)| &=\left|\sum_{e\in \Eh}\int_{e} \left(  \nabla u  -\mathcal{P}_{e}^{k-1} ( \nabla u) \right) \cdot \jump{v_h} \dS \right| &&\nonumber\\
&=\left|\sum_{e\in \Eh}\int_{e} \left(  \nabla u  -\mathcal{P}_{e}^{k-1} ( \nabla u) \right) \cdot \left( \jump{v_h}-\mathcal{P}_e^{0}( \jump{v_h} )\right) \dS \right| &&\nonumber\\
&\leq \sum_{e\in \Eh} \left\| \nabla u  -\mathcal{P}_{e}^{k-1} ( \nabla u) \right\|_{0,e} \left\| \jump{v_h}-\mathcal{P}_e^{0}( \jump{v_h})\right\|_{0,e}\;, &&
\end{align}
where $\mathcal{P}^{\ell}_{e}:L^{2}(e) \lor \mathbb{P}^{\ell}(e)$ is the $L^{2}$-orthogonal projection onto the space $\mathbb{P}^{\ell}(e)$ for $\ell\ge0$. 

Using now  standard approximation estimates (see \cite{ciar2}) we have for each $e=\partial\K^{+}\cap\partial\K^{-1}$,
\begin{align*}
\left\|  \nabla u  -\mathcal{P}_{e}^{k-1} (\nabla u) \right\|_{0,e} & \leq C  h^{\min(s,k)-1/2} \sqrt{ \|u\|_{s+1,\K^{+}\cup \K^{-}}}\;, &&\\
\left\|  \jump{v_h}-\mathcal{P}_e^{0}(\jump{v_h})    \right\|_{0,e} &\leq Ch^{1/2} \| \nabla v_h\|_{0,\K^{+}\cup \K^{-}}\;. &&
\end{align*}
Hence, substituting the above estimates into \eqref{eq:estima:Nh} and summing over all elements, the proof is concluded.
\end{proof}

\begin{remark}
  To obtain at least an estimate of first order of the term
  $\mathcal{N}_h(u,v)$, notice that the proof of Lemma \ref{le:Nh}
  requires further regularity (at least $u\in H^{2}(\O)$) than the one
  that problem \eqref{eq:mod00:A}-\eqref{eq:mod00:B} might have (as
  for instance in the case $f \in H^{-1}(\O)$ or even $f \in
  L^{2}(\O)$ and the domain not convex or with a second order problem
  with a jumping coefficient $\mathbb{K}$). We have followed the
  classical line for the error analysis to keep the presentation of
  the method simpler. Of course one might consider the extension of
  the results in \cite{gudi} to estimate the nonconformity error
  arising in the nonconforming virtual approximation. We wish to note
  though, that such extension will require to have laid for virtual
  elements, some results on a-posteriori error estimation. While that
  would be surely possible and it might merit further investigation,
  it is out of the scope of this paper and we feel that by sticking to
  the present proof, we are able to convey in a better way (and with a
  neat presentation) the novelty and new idea of the paper.
\end{remark}

We have the following result.

\begin{teo}\label{teo0}
Let  {\bf (A0)} be satisfied and let $u$ be the solution of \eqref{eq:var00}. Consider the nonconforming virtual element method in \eqref{eq:var0h}, 
with  $V_h^{k}$   given in \eqref{eq:def:nvem0} and with $a_h(\cdot,\cdot)$ and $f_h \in (V_h^{k})^{'}$ defined as in Section \ref{sec:2}.  Then, problem \eqref{eq:var0h} has a unique solution  $u_h\in V_h^{k}$. Moreover, for every approximation $u^{I}\in V_h^{k}$ of $u$ and for every piecewise polynomial approximation $u_{\pi}\in \mathbb{P}^{k}(\Th)$ of $u$, there exists a constant $C>0$ depending only on $\alpha_{\ast}$ and $\alpha^{\ast}$ in \eqref{eq:stability} such that the following estimate holds
\begin{equation}\label{eq:estima:h1}
|u-u_h|_{1,h} \leq C( |u-u^{I}|_{1,h}+ |u-u_{\pi}|_{1,h}   + \sup_{v_h \in V_h^{k}}\frac{|<f-f_h,v_h>|}{|v_h|_{1,h}}+ \sup_{v_h \in V_h^{k}}\frac{\mathcal{N}_h(u,v_h)}{|v_h|_{1,h}})\;.
\end{equation}
Furthermore, if $f\in H^{s-1}(\O)$  with $s\geq 1$, then we also have
\begin{equation}\label{eq:estima44}
|u-u_h|_{1,h} \leq C h^{\min{(k,s)}} ( \|u\|_{1+s,\O}   +\|f\|_{s-1,\O} )\;.
\end{equation}

\end{teo}

\begin{proof}
We first establish the existence and uniqueness of the solution to \eqref{eq:var0h}. 
From \eqref{eq:cont0}, \eqref{eq:stability} and \eqref{eq:norm-broken} we easily have coercivity and continuity of the global discrete bilinear form in ${H}^{1,nc}(\Th;k)$ (and in particular in $V_h^{k} \subset {H}^{1,nc}(\Th;k)$),
\begin{equation}\label{eq:stab:cont:global}
\begin{aligned}
a_h(v,v) &\geq  \alpha_{\ast} a(v,v)\geq C_{s} \alpha_{\ast} |v|_{1,h}^{2}\quad &&\forall\, v \in \Hnc, \\
|a_h(u,v) |&\leq \alpha^{\ast}  |u|_{1,h}|v|_{1,h} \quad &&\forall\, u,v \in \Hnc.
\end{aligned}
\end{equation}
 With $f_h \in (V_h^{k})^{'}$ and the Poincar\`e inequality \eqref{eq:poinc:02},  a direct application of Lax-Milgram theorem guarantees existence and uniqueness of the solution $u_h\in V_h^{k}$ of \eqref{eq:var0h}.\\

We now prove the error estimate. We first write $u-u_h=(u-\ui)+(\ui-u_h)$ and use triangle inequality to bound
\begin{equation}\label{eq:error-2}
|u-u_h|_{1,h} \leq |u-\ui |_{1,h}+|u_h-\ui|_{1,h}\;.
\end{equation}
The first term can be estimated   using the standard approximation \eqref{eq:interp:00} and so it is enough to estimate the second term on the right hand side above.
Let $\delta_h=u_h-\ui \in V_h^{k}$. Using the continuity \eqref{eq:cont0} and the $k$-consistency several times
\begin{align}\label{eq:04}
\alpha_{\ast} |\delta_h|^{2}_{1,h}  &=\alpha_{\ast} a(\delta_h,\delta_h) \leq a_h(\delta_h,\delta_h) &&\nonumber\\
&=a_h(u_h,\delta_h)-a_h(\ui,\delta_h) &&\nonumber\\
&=(f_h,\delta_h) -\sum_{\K\in \Th} a_h^{\K}(\ui-u_{\pi},\delta_h) -\sum_{\K\in \Th} a_h^{\K}(u_{\pi},\delta_h) &&\nonumber\\
&=(f_h,\delta_h) -\sum_{\K\in \Th} a_h^{\K}(\ui-u_{\pi},\delta_h) -\sum_{\K\in \Th} a^{\K}(u_{\pi},\delta_h) &&\nonumber\\
&=(f_h,\delta_h) -\sum_{\K\in \Th} a_h^{\K}(\ui-u_{\pi},\delta_h) +\sum_{\K\in \Th} a^{\K}(u-u_{\pi},\delta_h) -a(u,\delta_h) &&\nonumber\\
&=(f_h,\delta_h) -a(u,\delta_h)  -\sum_{\K\in \Th} a_h^{\K}(\ui-u_{\pi},\delta_h) +\sum_{\K\in \Th} a^{\K}(u-u_{\pi},\delta_h) &&\nonumber\\
&=(f_h,\delta_h) -(f,\delta_h) -\mathcal{N}_h(u,\delta_h)  -\sum_{\K\in \Th} a_h^{\K}(\ui-u_{\pi},\delta_h) &&\nonumber\\
&\phantom{=}
+\sum_{\K\in \Th} a^{\K}(u-u_{\pi},\delta_h)
\end{align}
where in the last step we have used \eqref{eq:a:nonc0} to introduce the {\it consistency error}. The proof is then concluded by estimating each of the terms in the right hand side above and substituting in \eqref{eq:error-2}.  Last part of the theorem, follows by using Lemma \ref{le:Nh} and the approximation estimates  \eqref{eq:approx:pi} and \eqref{eq:interp:00}  to bound  the terms on the right hand side of  \eqref{eq:estima:h1}.
\end{proof}

\begin{remark}
Theorem \ref{teo0} is the corresponding abstract result to \cite[Theorem 3.1]{volley}. 
As commented before, the term $\calN_h$ measures the extent to which the continuous solution $u$ fails to satisfy the {\it virtual element} formulation \eqref{eq:var0h}; measures the non-conformity of the approximation. In this respect, this result could be regarded as the analog for the VEM of the Strang Lemma for classical FEM.
\end{remark}

\subsection{$L^{2}(\O)$-error analysis}

We now report the $L^{2}$ error analysis of the proposed nonconforming VEM. It follows closely the  $L^{2}$-error analysis for classical nonconforming methods.

\begin{teo}\label{teo:02}
Let $\O$ be a convex domain and let $\Th$ be a family of partitions of $\Omega$ satisfying \ASSUM{A}{1}-\ASSUM{A}{3}. Let $k\ge1$ and let $u\in H^{s+1}(\O),\,\, s\geq 1$ be the solution of \eqref{eq:var00} and let $u_h\in V_h^{k}$ be its nonconforming virtual element approximation solving \eqref{eq:var0h}.  Then, there exists a positive constant $C$ depending on $k$, the regularity of the mesh and the shape of the domain such that
\begin{align}
  \|u-u_h\|_{0,\Th} 
  &\leq C h  (  |u-u_h |_{1,h}+|u-u_{\pi}\|_{1,h})+C(h^{2} +h^{\min{(2,\bar{k}+1)}} )\|f-f_h\|_{0,\O}\nonumber\\
  &\phantom{\leq} +C  h^{\min{(k,s)}+1}\|u\|_{s+1,\O}\;.
\end{align}
where $\bar{k}=\max\{k-2,0\}$.
\end{teo}


\begin{proof}
We consider the dual problem: find $\psi \in H^{2}(\O)\cap H^{1}_{0}(\O) $ solution of
\begin{equation*}
-\Delta\psi =u-u_h \quad \mbox{   in  } \O, \qquad \psi=0 \quad \mbox{ on  } \partial\Omega\;.
\end{equation*}
From the assumptions on the domain, the elliptic regularity theory gives the inequality $
\|\psi\|_{2,\Omega} \leq C \|u-u_h\|_{0,\Omega}$ 
where $C$ depends on the domain only through the domain's shape.  Let $\psi^{I}\in V_h^{k}$ and $\psi_{\pi}\in \mathbb{P}^{k}(\Th)$ be the approximations to $\psi$ 
satisfying 
\eqref{eq:interp:00} and \eqref{eq:approx:pi}.Then, integrating by parts we find
\begin{align}\label{eq:inter:04}
\|u-u_h\|_{0,\Th}^{2} &= \int_{\O} -\Delta\psi (u-u_h) d\bx &&\nonumber\\
&= \sum_{\K\in \Th} \int_{\K}\nabla \psi \cdot \nabla (u-u_h) d\bx +\sum_{\K\in \Th} \int_{\partial \K} \frac{\partial\psi}{\partial \n} \,(u-u_h) \dS &&\nonumber\\
&=a( \psi-\psi^{I}, (u-u_h) ) + a( \psi^{I}, (u-u_h) )  +\calN_{h}(\psi, u-u_h)\;.
\end{align}
We now estimate the three terms above. 
The estimate for the first one follows from the continuity of $a(\cdot,\cdot)$ together with the approximation properties \eqref{eq:interp:00} of $\psi^{I}$ and the a-priori estimate of $\psi$
\begin{equation*}
|a( \psi-\psi^{I}, u-u_h )|\leq C| \psi-\psi^{I}|_{1,h} |u-u_h|_{1,h}\leq C h\|u-u_h\|_{0,\Th}  |u-u_h |_{1,h}\;.
\end{equation*}
Last term is readily estimated by means of Lemma \ref{le:Nh} with $k=s=1$ (since obviously $u-u_h\in H^{1,nc}(\Th;1)$), giving
\begin{equation*}
|\calN_{h}(\psi, u-u_h)| \leq Ch\|\psi\|_{2,\O} |u-u_h|_{1,h}\leq C h\|u-u_h\|_{0,\Th}  |u-u_h |_{1,h}\;.
\end{equation*}
To estimate the second term in \eqref{eq:inter:04} we use the symmetry of the problem together with  \eqref{eq:var0h} and \eqref{eq:a:nonc0} to write
\begin{align}
a( \psi^{I}, u-u_h )
&= a( u,\psi^{I})- a(u_h,\psi^{I} )  &&\nonumber\\
&=\calN_h(u,\psi^{I}) +\langle f,\psi_i\rangle - a(u_h,\psi^{I} )
+ a_h(u_h,\psi^{I} ) - a_h(u_h,\psi^{I} ) &&\nonumber\\
&=\calN_h(u,\psi^{I}) +\langle f -f_h,\psi^{I} \rangle +\big(a_h(u_h,\psi^{I})-a(u_h,\psi^{I} )\big)  &&\nonumber\\
&=T_0+T_1+T_2&&\label{eq:tes0}
\end{align}
To conclude we need to estimate each of the above terms. For the first one, we first notice that from the definition \eqref{eq:def:Nh} and the regularity of $\psi$, one obviously has $\calN_h(u,\psi)=0$. Hence, a standard application of Lemma \ref{le:Nh} together with the approximation properties \eqref{eq:interp:00} of $\psi^{I}$ and the a-priori estimate of $\psi$, gives
\begin{align}
|T_0|
&=|\calN_h(u,\psi^{I})|=|\calN_h(u,\psi^{I}-\psi)| \leq Ch^{\min{(k,s)}}\|u\|_{s+1,\O} |\psi^{I}-\psi|_{1,h} &&\nonumber\\
&\leq C h^{\min{(k,s)}+1}\|u\|_{s+1,\O}\|u-u_h\|_{0,\Th}\;.
\end{align}
The last two terms in \eqref{eq:tes0} can be bounded as in \cite{vem-elas}. Here, we report the proof for the sake of completeness. 
For $T_1$, using the $L^{2}$-orthogonal projection, and denoting again $\bar{k}=\max\{k-2,0\}$, we find
\begin{align}
T_1&= \sum_{\K\in\Th}\left( \int_{\K} (f-f_h) (\psi^{I}-\psi) dx + \int_{\K} (f-f_h) (\psi-\mathcal{P}^{\bar{k}}_{\K} (\psi)) dx\right) &&\nonumber\\
&\qquad \quad \leq \|f-f_h\|_{0,\O}(\|\psi^{I}-\psi\|_{0,\Th} +\|\psi -\mathcal{P}^{\bar{k}}_{\K} (\psi)\|_{0,\Th}) &&\nonumber\\
&\qquad \quad \leq C(h^{2} +h^{\min{(2,\bar{k}+1)}} )\|f-f_h\|_{0,\O}\|u-u_h\|_{0,\O}\;. &&
\end{align}

As regards $T_2$, we use the symmetry together with the $k$-consistency property twice, and the definition of the norm \eqref{eq:norm-broken}
\begin{align*}
T_{2} &=a_h(u_h,\psi^{I})-a(u_h,\psi^{I} ) =\sum_{\K\in \Th} \big(a^{\K}_h(u_h-u_{\pi},\psi^{I})-a^{\K}(u_h-u_{\pi},\psi^{I} )\big)  &&\nonumber\\
&=\sum_{\K\in \Th} \big(a^{\K}_h(u_h-u_{\pi},\psi^{I}-\psi_{\pi})-a^{\K}(u_h-u_{\pi},\psi^{I}-\psi_{\pi} )\big) && \nonumber\\
&\leq | u_h-u_{\pi}|_{1,h} |\psi_h-\psi_{\pi}|_{1,h}\;.
\end{align*}
Each of the above terms can be readily estimated by  adding and subtracting $u$ and $\psi$:
\begin{align*}
 | u_h-u_{\pi}|_{1,h}^{2} &\leq \sum_{\K\in \Th}  \left(\|\nabla(u_h-u)\|_{0,\K}^{2} +\| \nabla(u-u_{\pi})|_{0,\K}^{2} \right) &&\\
 |\psi_h-\psi_{\pi}|_{1,h}^{2}  &\leq \sum_{\K\in \Th} \left(\|\nabla(\psi_h-\psi)\|_{0,\K}^{2} +\| \nabla(\psi-\psi_{\pi})|_{0,\K}^{2} \right) \leq Ch^{2} \|u-u_h\|^{2}_{0,\Th}\;, &&
 \end{align*}
 where in the last step we have also used the standard approximation properties \eqref{eq:approx:pi} and \eqref{eq:interp:00}. With the above estimates, the bound for the term $T_2$ finally reads
 \begin{equation*}
 T_2 \leq Ch \|u-u_h\|_{0,\Th} (|u-u_h|_{1,h}+|u-u_{\pi}|_{1,h})
 \end{equation*}
Plugging now the estimates for $T_0,T_1$ and $T_2$ into \eqref{eq:tes0} we finally get:
\begin{align*}
\|u-u_h\|_{0,\Th}
&\leq C h (  h^{\min{(k,s)}+1}\|u\|_{s+1,\O}+ |u-u_h |_{1,h}+|u-u_{\pi}|_{1,h})+C(h^{2} \\
&\phantom{\leq} +h^{\min{(2,\bar{k}+1)}} )\|f-f_h\|_{0,\O}\;,
\end{align*}
which concludes the proof.

\end{proof}

\section{Connection with the nonconforming MFD \\method \cite{Lipnikov-Manzini:2014:JCP}}
\label{sec:4}
In this section, we discuss relationships between the proposed
nonconforming VEM and the nonconforming MFD method in~\cite{Lipnikov-Manzini:2014:JCP}.
Throughout this section we will use the notation of~\cite{BeiraodaVeiga-Brezzi-Marini-Russo:2014}.
Also, we will omit the element index $\K$ from all matrix symbols.

The stiffness matrix $\matM^{\VEM{}}$ of the nonconforming VEM is formally defined as
\begin{align*}
  a^{\K}_{h}(u_h,v_h) = {\bf v}_h^T\matM^{\VEM{}}{\bf u}_h,
\end{align*}
where ${\bf v}_{h}$ and ${\bf u}_{h}$ are algebraic vectors collecting the degrees of
freedom of functions $v_h$ and $u_h$, respectively.
We enumerate the whole set of $\npoly$ scaled monomials used in~\eqref{eq:dofs:01}
and~\eqref{eq:dofs:02} to define the degrees of freedom by local indices
$i$ and $j$ (resp., $\mm_i$ and $\mm_j$) ranging from $1$ to $\npoly$.

To compute the stiffness matrix, we need two auxiliary matrices
$\matB$ and $\matD$.
The $j$-th column of matrix $\matB$, for $j=1,\ldots,n_\K$, is defined by
\begin{align}
  \matB_{1j} &= 
  \begin{cases}
    \displaystyle\int_{\partial\K}\psi_j \,\dS&=0 \qquad\textrm{if~}k=1,\\[1.em]
    \displaystyle\int_{\K}       \psi_j\dx&=0 \qquad\textrm{if~}k\ge 2,
  \end{cases}\\[1em]
  \matB_{ij} &= \int_{\K}\nabla\mm_i\cdot\nabla\psi_j\dx,
  \quad
  \,i=2,\ldots,\npoly.
  \label{eq:matB}
\end{align}
The $j$-th column of matrix $\matD$, for
$j=1,\ldots,N_{\K}$, collects the degrees of freedom of the $j$-th
monomials and is defined by:
\begin{align}
  \matD_{ij} = \chi_i(\mm_j),\qquad i=1,\ldots,n_\K.
  \label{eq:matD}
\end{align}

Now, we consider the matrices $\matG=\matB\matD$,
$\bPIN=\matD\matG^{-1}\matB$ and $\matGt$, which is obtained from
matrix $\matG$ by setting its first row to zero.
The VEM stiffness matrix is the sum of two matrices, 
$\matM^{\VEM{}} = \matM^{\VEM{}}_{0} + \matM^{\VEM{}}_{1}$, which are
defined by the following formula:
\begin{align}
  \matM^{\VEM{}}
  = (\matG^{-1}\matB)^T\matGt(\matG^{-1}\matB) + (\matI-\bPIN)^T\matS(\matI-\bPIN),
  \label{eq:VEM:stiffness:matrix}
\end{align}
where $\matI$ is the identity matrix and $\matS$ is the matrix
representation of the bilinear form $S^{\K}$.
The first matrix term corresponds to the consistency property and the
second term ensures stability.
According to \eqref{eq:def:gamma}, we can set
\begin{align}
  \matS = h^{d-2}\matI.
\end{align}
Since the choice of $S^{\K}$ is not unique, so is the choice of $\matS$; 
therefore, we have a family of virtual element schemes that differ
by matrix $\matS$.

The mimetic stiffness matrix considered
in~\cite{Lipnikov-Manzini:2014:JCP} has the same structure, 
$\matM^{\MFD{}} = \matM^{\MFD{}}_{0} + \matM^{\MFD{}}_{1}$, and the
two matrices $\matM^{\MFD{}}_{0}$ and $\matM^{\MFD{}}_{1}$ are also
related to the consistency and stability properties.
In particular, matrix $\matM^{\MFD{}}_{1}$ is given by:
\begin{align}
  \matM^{\MFD{}}_{1} = \big(\matI-\bPIp\big)\matU\big(\matI-\bPIp\big),
  \label{eq:MFD:stab:matrix}
\end{align}
where $\bPIp=\matD(\matD^T\matD)^{-1}\matD^T$ is the orthogonal
projector on the linear space spanned by the columns of matrix $\matD$
and $\matU$ is a symmetric and positive definite matrix of parameters.

Since both the VEM and the MFD method use the same degrees of freedom,
they must satisfy the same conditions of consistency and
stability.
Moreover, the matrices $\matM^{\MFD{}}_{0}$ and $\matM^{\VEM{}}_{0}$
are uniquely determined by the consistency condition (the exactness property 
on the same set of polynomials of degree $k$); thus, they must coincide.
Consequently, the virtual and mimetic stiffness matrices may differ
only for the stabilization terms $\matM^{\VEM{}}_{1}$ and
$\matM^{\MFD{}}_1$.
The relation between $\matM^{\VEM{}}_{1}$ and $\matM^{\MFD{}}_1$ is
established by the following lemma.
\begin{lemma}
  \label{lemma:MFD-vs-VEM}
  \begin{description}
  \item[$(i)$] For any mimetic stabilization matrix of the
    form~\eqref{eq:MFD:stab:matrix}, we can find a matrix $\matS$ such
    that $\matM^{\VEM{}}_{1}$ and $\matM^{\MFD{}}_{1}$ coincide.
    \medskip
  \item[$(ii)$] For any virtual element stabilization matrix as the second
    term in the right-hand-side of~\eqref{eq:VEM:stiffness:matrix}, we can
    find a matrix $\matU$ such that $\matM^{\MFD{}}_{1}$ and
    $\matM^{\VEM{}}_{1}$ coincide.
  \end{description}
\end{lemma}
\begin{proof}
  $(i)$ A straightforward calculation shows that
  \begin{align}
    \label{eq:rel:1}
    \bPIN\bPIp=\bPIp,
    \quad \big(\bPIN\big)^T\bPIp=\big(\bPIN\big)^T, 
    \quad \bPIp\bPIN=\bPIN.
  \end{align}
  We take $\matS=\matM^{\MFD{}}_{1}$. Using \eqref{eq:rel:1} yields:
  \begin{align*}
    \matM^{\VEM{}}_1 
    = \big(\matI-\bPIN\big)^T\,\big(\matI-\bPIp\big)\matU\big(\matI-\bPIp\big)\,(\matI-\bPIN) 
    = \matM^{\MFD{}}_1.
  \end{align*}

  \medskip\noindent
  $(ii)$ The relations in~\eqref{eq:rel:1} imply that
  $\big(\matI-\bPIp\big)\,\big(\matI-\bPIN\big)^T = \big(\matI-\bPIN\big)^T$.
  The assertion of the lemma follows by taking
  \begin{align*}
    \matU = (\matI-\bPIN)^T\matS(\matI-\bPIN) = \matM^{\VEM{}}_{1}.
  \end{align*}
\end{proof}

\medskip
\begin{remark}
  An effective and practical choice in the mimetic technology 
  (see \cite{Lipnikov-Manzini:2014:JCP}) is
  provided by taking $\matU=\rho\matI$ where $\rho$ is a
  scaling factor defined as the mean trace of $\matM^{\MFD{}}_0$.
  This implies that $\matM^{\MFD{}}_{1} = \rho\big(\matI - \bPIp\big)$.
\end{remark}

\section{Conclusions}
\label{sec:5}

In this work, we introduced the non-conforming virtual element method (VEM) for an
elliptic equation.
The VEM allows us to built arbitrary order schemes
on shape-regular polygonal and polyhedral meshes that may include
non-convex and degenerate elements.
In contrast to the  classical non-conforming finite element methods,
the construction of the VEM is done at once for any degree
$k \ge 1$ and any element shape.
Another advantage of the virtual element framework is ability to carry
out theoretical analysis for complex meshes reusing many existing
functional analysis tools.
We have shown the optimal convergence estimates in the energy and $L^2$ norms.
We also established an algebraic equivalence of the VEM and the mimetic
finite difference method from \cite{Lipnikov-Manzini:2014:JCP}.

 \section*{Aknowledgements}

The first author is in-debt with Franco Brezzi and Donatella Marini
from Pavia, for the multiple and fruitful discussions and specially
for the encouragement  to carry out  this work. She also thanks the
IMATI-CNR at Pavia, where most part of this work was done, while she
was visiting in December 2013.
The work of the second author was partially supported by the DOE
Office of Science Advanced Scientific Computing Research (ASCR)
Program in Applied Mathematics Research.
The work of the third author was partially supported by the LANL
LDRD-ER Project \#20140270 and by IMATI-CNR, Pavia, Italy.


\end{document}